\documentclass{article}%
\usepackage{amsmath}
\usepackage{amsfonts}
\usepackage{amssymb}
\usepackage{graphicx}%
\providecommand{\U}[1]{\protect\rule{.1in}{.1in}}
\newtheorem{theorem}{Theorem}

\newtheorem{corollary}[theorem]{Corollary}

\newtheorem{definition}[theorem]{Definition}
\newtheorem{example}[theorem]{Example}

\newtheorem{lemma}[theorem]{Lemma}

\newtheorem{proposition}[theorem]{Proposition}
\newtheorem{remark}[theorem]{Remark}

\newenvironment{proof}[1][Proof]{\noindent\textbf{#1.} }{\ \rule{0.5em}{0.5em}}
\begin{document}

\title{On the Definitions of Fractional Sum and Difference on Non-uniform Lattices}
\author{Jinfa Cheng}
\date{{\small School of Mathematical Sciences, Xiamen University, }\\
{\small Xiamen, Fujian, 361005, P. R. China}\\
{\small E-mail: jfcheng@xmu.edu.cn}}
\maketitle

\begin{abstract}
As is well known, the idea of a fractional sum and difference on uniform
lattice is more current, and gets a lot of development in this field. But the
definitions of fractional sum and fractional difference of $f(z)$ on
non-uniform lattices $x(z)=c_{1}z^{2}+c_{2}z+c_{3}$ or $x(z)=c_{1}q^{z}%
+c_{2}q^{-z}+c_{3}$ seem much more difficult and complicated. In this article,
for the first time we propose the definitions of the fractional sum and
fractional difference on non-uniform lattices by two different ways. The
analogue of Euler's Beta formula, Cauchy' Beta formula on on non-uniform
lattices are established, and some fundamental theorems of fractional
calculas, the solution of the generalized Abel equation and fractional central
difference equations on non-uniform lattices are obtained etc.

\end{abstract}



\section{\bigskip Intrduction}

\bigskip The definition of non-uniform lattices date back to the approximation
of the following differential equation of hypergeometric type:%
\begin{equation}
\sigma(z)y^{\prime\prime}(z)+\tau(z)y^{\prime}(z)+\lambda
y(z)=0,\label{hytype}%
\end{equation}
where $\sigma(z)$ and $\tau(z)$ are polynomials of degrees at most two and
one, respectively, and $\lambda$ is a constant. Its solutions are some types
of special functions of mathematical physics, such as the classical orthogonal
polynomials, the hypergeometric and cylindrical functions, see G. E. Andrews,
R. Askey, R. Roy \cite{andrews1999} and Z. X. Wang \cite{wang1989}. A. F.
Nikiforov, V. B. Uvarov and S. K. Suslov \cite{nikiforov1991, nikiforov1988}
generalized Eq. (\ref{hytype}) to a difference equation of hypergeometric type
case and studied the Nikiforov-Uvarov-Suslov difference equation on a lattice
$x(s)$ with variable step size $\Delta x(s)=x(s+1)-x(s),\nabla
x(s)=x(s)-x(s-1)$ as%
\begin{equation}
\widetilde{\sigma}[x(s)]\frac{\Delta}{\Delta x(s-1/2)}\left[  \frac{\nabla
y(s)}{\nabla x(s)}\right]  +\frac{1}{2}\widetilde{\tau}[x(s)]\left[
\frac{\Delta y(s)}{\Delta x(s)}+\frac{\nabla y(s)}{\nabla x(s)}\right]
+\lambda y(s)=0,\label{NUSeq}%
\end{equation}
where $\widetilde{\sigma}(x)$ and $\widetilde{\tau}(x)$ are polynomials of
degrees at most two and one in $x(s),$ respectively, $\lambda$ is a constant,
$\Delta y(s)=y(s+1)-y(s),$ $\nabla y(s)=y(s)-y(s-1),$ and $x(s)$ is a lattice
function\ that satisfies
\begin{equation}
\frac{x(s+1)+x(s)}{2}=\alpha x(s+\frac{1}{2})+\beta,\text{\hspace{0.1in}%
}\alpha,\beta\text{ are constants,}\label{condition1}%
\end{equation}%
\begin{equation}
x^{2}(s+1)+x^{2}(s)\text{ is a polynomial of degree at most two w.r.t.
}x(s+\frac{1}{2}).\label{condition2}%
\end{equation}
It should be pointed out that the difference equation (\ref{NUSeq}) obtained
as a result of approximating the differential equation (\ref{hytype}) on a
non-uniform lattice is of independent importance and arises in a number of
other questions. Its solutions essentially generalized the solutions of the
original differential equation and are of interest in their own right
\cite{askey1979,askey1984,askey1985}.

\begin{definition}
(\cite{nikiforov1991,nikiforov1988})\label{def1} Two kinds of lattice
functions $x(s)$ are called \emph{non-uniform lattices} which satisfy the
conditions in Eqs. (\ref{condition1}) and (\ref{condition2}) are
\begin{equation}
x(s)=\widetilde{c}_{1}s^{2}+\widetilde{c}_{2}s+\widetilde{c}_{3},\label{non2}%
\end{equation}%
\begin{equation}
x(s)=c_{1}q^{s}+c_{2}q^{-s}+c_{3},\label{non1}%
\end{equation}
where $c_{i},\widetilde{c}_{i}$ are arbitrary constants and $c_{1}c_{2}\neq0$,
$\widetilde{c}_{1}\widetilde{c}_{2}\neq0$. When $c_{1}=1,c_{2}=c_{3}=0,$ $or$
$\widetilde{c}_{2}=1,\widetilde{c}_{1}=\widetilde{c}_{3}=0,$ these two kinds
of lattice functions $x(s)$%
\begin{equation}
x(s)=s\label{uni2}%
\end{equation}%
\begin{equation}
x(s)=q^{s},\label{uni1}%
\end{equation}
are called \emph{uniform lattices.}
\end{definition}

Let $x(s)$ be a non-uniform lattice, where $s\in\mathbb{C}$. For any real
$\gamma$, $x_{\gamma}(s)=x(s+\frac{\gamma}{2})$ is also a non-uniform lattice.
Given a function $F(s)$, define the difference operator with respect to
$x_{\gamma}(s)$ as%

\[
\nabla_{\gamma}F(s)=\frac{\nabla F(s)}{\nabla x_{\gamma}(s)},
\]
and%
\[
\nabla_{\gamma}^{k}F(z)=\frac{\nabla}{\nabla x_{\gamma}(z)}(\frac{\nabla
}{\nabla x_{\gamma+1}(z)}...(\frac{\nabla F(z)}{\nabla x_{\gamma+k-1}%
(z)})).(k=1,2,...)
\]

The following equalities can be verified straightforwardly.

\begin{proposition}
\label{pro}\textit{Given two functions} $f(s),g(s)$ \textit{with complex
variable} $s$,\textit{\ the following difference equalities hold}%
\begin{align*}
\Delta_{\nu}(f(s)g(s))  &  =f(s+1)\Delta_{\nu}g(s)+g(s)\Delta_{\nu}f(s)\\
&  =g(s+1)\Delta_{\nu}f(s)+f(s)\Delta_{\nu}g(s),
\end{align*}%
\begin{align*}
\Delta_{\nu}\left(  \frac{f(s)}{g(s)}\right)   &  =\frac{g(s+1)\Delta_{\nu
}f(s)-f(s+1)\Delta_{\nu}g(s)}{g(s)g(s+1)}\\
&  =\frac{g(s)\Delta_{\nu}f(s)-f(s)\Delta_{\nu}g(s)}{g(s)g(s+1)},
\end{align*}%
\begin{align}
\nabla_{\nu}(f(s)g(s))  &  =f(s-1)\nabla_{\nu}g(s)+g(s)\nabla_{\nu
}f(s)\nonumber\\
&  =g(s-1)\nabla_{\nu}f(s)+f(s)\nabla_{\nu}g(s), \label{Leibnz}%
\end{align}%
\begin{align*}
\nabla_{\nu}\left(  \frac{f(s)}{g(s)}\right)   &  =\frac{g(s-1)\nabla_{\nu
}f(s)-f(s-1)\nabla_{\nu}g(s)}{g(s)g(s-1)}\\
&  =\frac{g(s)\nabla_{\nu}f(s)-f(s)\nabla_{\nu}g(s)}{g(s)g(s-1)}.
\end{align*}

\end{proposition}

The notions of fractional calculus date back to Euler, and in the last decades
the fractional calculas had a remarkable development as shown by many
mathematical volumes dedicated to it
\cite{oldham1974,miller1993,samko1993,kiryakova1994,podlubny1999,diethelm2002,kilbas2006}%
, but the idea of a fractional difference on uniform lattice (\ref{uni2}) and
(\ref{uni1}) is more current. $\ $

Some of the more extensive papers on the fractional difference on uniform
lattice (\ref{uni2}), Diaz and Osler \cite{diaz1974}, Granger and Joyeux
\cite{granger1980}, Hosking \cite{hosking1981} have employed the definition of
the $\alpha$-th order fractional difference by

\bigskip%
\begin{equation}
\nabla^{\alpha}f(x)=\sum_{k=0}^{\infty}(-1)^{k}\left(
\begin{array}
[c]{c}%
\alpha\\
k
\end{array}
\right)  f(x-k),
\end{equation}
where $\alpha$ is any real number and the notation $\nabla^{\alpha}$ is used
since this definition is natural extension of the backward difference operator.

H. H. Gray and N. F. Zhang \cite{gray1988} gave the following new definition
of the fractional sum and difference:

\begin{definition}
(\cite{gray1988})\label{def3} For $\alpha$ any complex number, and $f$ defined
over the integer set $\{a,a+1,...,x\}$, the $\alpha$-th order sum over
$\{a,a+1,...,x\}$ is defined by%
\begin{equation}
S_{a}^{\alpha}f(x)=\frac{1}{\Gamma(\alpha)}\sum_{k=a}^{x}(x-k+1)_{\alpha
-1}f(k).
\end{equation}

\end{definition}

For any complex number $\alpha$ and $\beta$ let $(\alpha)_{\beta\text{ }}$be
difined as follows:%

\[
(\alpha)_{\beta}=\left\{
\begin{array}
[c]{c}%
\frac{\Gamma(\alpha+\beta)}{\Gamma(\alpha)},\text{ \ \ \ when }\alpha\text{
and }\beta\text{ are neither zero nor negative integers,}\\
1,\text{
\ \ \ \ \ \ \ \ \ \ \ \ \ \ \ \ \ \ \ \ \ \ \ \ \ \ \ \ \ \ \ \ \ \ \ \ \ \ when
}\alpha=\beta=0,\\
0,\text{ \ \ \ \ \ \ \ \ \ \ \ when }\alpha=0,\beta\text{ is not zero or a
negative integer}\\
\text{undefined
\ \ \ \ \ \ \ \ \ \ \ \ \ \ \ \ \ \ \ \ \ \ \ \ \ \ \ \ \ \ \ \ \ \ \ \ \ \ \ \ \ \ \ otherwise}%
\end{array}
\right.
\]
and when $n\in N,(\alpha)_{n}=a(a+1)...(a+n-1)$ denotes the Pochhammer symbol.

\begin{definition}
(\cite{gray1988})\label{def4} For $\alpha$ any complex number, the $\alpha
$th-order difference over $\{a,a+1,...,x\}$ is defined by%
\begin{equation}
\nabla_{a}^{\alpha}f(x)=S_{a}^{-\alpha}f(x).
\end{equation}

\end{definition}

J. F. Cheng \cite{jinfa2011} independently gave the following definitions of
the fractional sum and difference, which are consistent with Definition
\ref{def3} and Definition \ref{def4}, and are well defined for any real or
complex number $\alpha.$

\begin{definition}
(\cite{jinfa2011})\label{def5} For $\alpha$ complex number, $Re \alpha>0$ and
$f$ defined over the integer set $\{a,a+1,...,x\}$, the $\alpha$-th order sum
over $\{a,a+1,...,x\}$ is defined by%
\begin{equation}
\nabla_{a}^{-\alpha}f(x)=\sum_{k=a}^{x}\left[
\begin{array}
[c]{c}%
\alpha\\
x-k
\end{array}
\right]  f(k),
\end{equation}
where $\left[
\begin{array}
[c]{c}%
\alpha\\
n
\end{array}
\right]  =\frac{\alpha(\alpha+1)...(\alpha+n-1)}{n!}.$
\end{definition}

\begin{definition}
(\cite{jinfa2011})\label{def6} For $\alpha$ any complex number, $n-1\leq
\operatorname{Re}\alpha<n$, the Riemann-Liouville type $\alpha$-th order
difference over $\{a,a+1,...,x\}$ is defined by%
\begin{equation}
\nabla_{a}^{\alpha}f(x)=\nabla^{n}\nabla_{a}^{\alpha-n}f(x),
\end{equation}
and Caputo type $\alpha$th-order difference over $\{a,a+1,...,x\}$ is defined by
\end{definition}

\begin{equation}
\nabla_{a}^{\alpha}f(x)=\nabla_{a}^{\alpha-n}\nabla^{n}f(x).
\end{equation}

In the case of uniform lattice uniform (\ref{uni1}), a q-analogue of the
Riemann-Liouville fractional sum operator is introduced in \cite{salam1966} by
Al-Salam through%

\begin{equation}
I_{q}^{\alpha}f(x)=\frac{x^{\alpha-1}}{\Gamma_{q}(\alpha)}\int_{0}%
^{x}(qt/x;q)_{\alpha-1}f(t)d_{q}(t).
\end{equation}

The $q$-analogue of the Riemann-Liouville fractional difference operator is
also given independently by Agarwal \cite{agarwal1969}, who defined the
q-fractional difference to be%

\begin{equation}
D_{q}^{\alpha}f(x)=I_{q}^{-\alpha}f(x)=\frac{x^{-\alpha-1}}{\Gamma_{q}%
(-\alpha)}\int_{0}^{x}(qt/x;q)_{-\alpha-1}f(t)d_{q}(t).
\end{equation}

Althought the discrete fractional calculus on uniform lattice (\ref{uni2}) and
(\ref{uni1}) are more current, but great development has been made in this
field \cite{atici2009,eloe2009,anastassiou2010,torres2011,baoguo2016}. In the
recent monographs, J. F. Cheng \cite{jinfa2011}, C. Goodrich and A. Peterson
\cite{goodrich2015} provided the comprehensive treatment of the discrete
fractional calculus with up-to-date references, and the developments in the
theory of fractional $q$-calculas had been well reported by M. H. Annaby and
Z. S. Mansour \cite{annaby2012}.

But we should mention that, in the case of nonuniform lattices (\ref{non2}) or
(\ref{non1}), even when $n\in N,$ the fomula of $n-$order difference on
non-uniform lattices is a remarkable job, since it is very complicated and
difficult to be obtained. In fact, in \cite{nikiforov1991,nikiforov1988}, A.
Nikiforov, V. Uvarov, S. Suslov obtained the formula of $n-th$ difference
$\nabla_{1}^{(n)}[f(s)]$ as follows:

\begin{definition}
(\cite{nikiforov1991,nikiforov1988})\label{def7} Let $n\in N^{+},$ for
nonuniform lattices (\ref{non2}) or (\ref{non1}), then%
\begin{align}
\nabla_{1}^{(n)}[f(s)]  &  =\sum_{k=0}^{n}\frac{(-1)^{n-k}[\Gamma(n+1)]_{q}%
}{[\Gamma(k+1)]_{q}[\Gamma(n-k+1)]_{q}}\nonumber\\
&  \times%
{\displaystyle\prod\limits_{l=0}^{n}}
\frac{\nabla x[s+k-(n-1)/2]}{\nabla x[s+(k-l+1)/2]}f(s-n+k)\nonumber\\
&  =\sum_{k=0}^{n}\frac{(-1)^{n-k}[\Gamma(n+1)]_{q}}{[\Gamma(k+1)]_{q}%
[\Gamma(n-k+1)]_{q}}\nonumber\\
&  \times%
{\displaystyle\prod\limits_{l=0}^{n}}
\frac{\nabla x_{n+1}(s-k)}{\nabla x[s+(n-k-l+1)/2]}f(s-k),
\end{align}
where $[\Gamma(s)]_{q}$ is modified $q-$gamma function which is defined as%
\[
\lbrack\Gamma(s)]_{q}=q^{-(s-1)(s-2)/4}\Gamma_{q}(s),
\]
and function $\Gamma_{q}(s)$ is called the $q-$gamma function; it is a
generalization of Euler's gamma function $\Gamma(s)$. It is defined by
\end{definition}

\begin{equation}
\Gamma_{q}(s)=\left\{
\begin{array}
[c]{c}%
\frac{\Pi_{k=0}^{\infty}(1-q^{k+1})}{(1-q)^{s-1}\Pi_{k=0}^{\infty}(1-q^{s+k}%
)},\text{ \ \ \ when }|q|<1;\\
q^{-(s-1)(s-2)/2}\Gamma_{1/q}(s),\text{ \ \ \ \ \ \ \ \ \ \ \ when }|q|<1.
\end{array}
\right.  \label{gammaq}%
\end{equation}

After further transformations, A. Nikiforov, V. Uvarov, S. Suslov in
\cite{nikiforov1991} rewritted the formula of $n-th$ difference $\nabla
_{1}^{(n)}[f(s)]$ as follows:

\begin{definition}
(\cite{nikiforov1991})\label{def8} Let $n\in N^{+},$ for nonuniform lattices
(\ref{non2}) or (\ref{non1}), then%
\[
\nabla_{1}^{(n)}[f(s)]=\sum_{k=0}^{n}\frac{([-n]_{q})_{k}}{[k]_{q}!}%
\frac{[\Gamma(2s-k+c)]_{q}}{[\Gamma(2s-k+n+1+c)]_{q}}f(s-k)\nabla
x_{n+1}(s-k),
\]
where%
\begin{equation}
\lbrack\mu]_{q}=%
\genfrac{\{}{.}{0pt}{}{\frac{q^{\frac{\mu}{2}}-q^{-\frac{\mu}{2}}}{q^{\frac
{1}{2}}-q^{-\frac{1}{2}}},\text{if }x(s)=c_{1}q^{s}+c_{2}q^{-s}+c_{3}%
;}{\mu,\text{if }x(s)=\widetilde{c}_{1}s^{2}+\widetilde{c}_{2}s+\widetilde{c}%
_{3},}
\label{muq}%
\end{equation}
and
\[
c=\left\{
\begin{array}
[c]{c}%
\frac{\log\frac{c_{2}}{c_{1}}}{\log q},\text{ \ \ \ when }x(s)=c_{1}%
q^{s}+c_{2}q^{-s}+c_{3},\\
\frac{\widetilde{c_{2}}}{\widetilde{c_{1}}},\text{ \ \ \ \ \ \ \ \ \ \ \ when
}x(s)=\widetilde{c}_{1}s^{2}+\widetilde{c}_{2}s+\widetilde{c}_{3}.
\end{array}
\right.
\]

\end{definition}

Now there exist two important and challenging problems that need to be further discussed:

(1) Assume that $g(s)$ be a known function, $f(s)$ be an unknown function,
which satisfies the following generalized difference equation on non-uniform lattices%

\begin{equation}
\nabla_{1}^{(n)}[f(s)]=g(s), \label{unidiff}%
\end{equation}
how to solve generalized difference equation (\ref{unidiff})?

(2) However, the related definitions of $\alpha-$order fractional sum and
$\alpha-$order fractional difference on non-uniform lattices are very
difficult and interesting problems, they have not been appeared since the
monographs \cite{nikiforov1991,nikiforov1988} were published. Can we give
reasonable definitions of fractional sum and difference on non-uniform Lattices?

We think that as the most general discrete fractional calculus on non-uniform
Lattices, they will have an independent meaning and lead to many interesting
new theories about them. They are the important extension and development of
\textbf{Definition} \ref{def7}, \ref{def8} and the discrete fractional calculus.

It is the purpose of this paper to inquires into the feasibility of
establishing a fractional calculus of finite difference on nonuniform
lattices. In this article, for the first time we propose the definitions of
the fractional sum and fractional difference on non-uniform lattices. Then
give some fundamental theorems of fractional calculas, such as the analogue of
Euler's Beta formula, Cauchy' Beta formula, Taylor's formula on non-uniform
lattices are established, and the solution of the generalized Abel equation
and fractional central difference equations on non-uniform lattices are
obtained etc. The resuts we obtain are essentially new and appeared in the
literature for the first time only recently.

\section{Integer sum and Fractional Sum on Non-uniform Lattices}

Let $x(s)$ be a non-uniform lattice, where $s\in\mathbb{C}$. let
$\nabla_{\gamma}F(s)=f(s).$ Then
\[
F(s)-F(s-1)=f(s)\left[  x_{\gamma}(s)-x_{\gamma}(s-1)\right]
\]
Choose $z,a\in\mathbb{C}$, and $z-a\in N$. Summing from $s=a+1$ to $z$, we
have
\[
F(z)-F(a)=\sum_{s=a+1}^{z}f(s)\nabla x_{\gamma}(s).
\]
Thus, we define
\[
\int_{a+1}^{z}f(s)d_{\nabla}x_{\gamma}(s)=\sum_{s=a+1}^{z}f(s)\nabla
x_{\gamma}(s).
\]
It is easy to verify that

\begin{proposition}
\label{integration}\mbox{} Given two function $F(z),f(z)$ with complex
variable $z,a\in C$, and $z-a\in N$, we have \begin{flalign*}
&(1)\nabla_{\gamma}\left[\int_{a+1}^{z}f(s)d_\nabla x_{\gamma}(s)\right]=f(z),&&\\
&(2)\int_{a+1}^{z}\nabla_{\gamma}F(s)d_\nabla x_{\gamma}(s)=F(z)-F(a).&&\\
\end{flalign*}

\end{proposition}

A generalized power ${[x(s)-x(z)]^{(n)}}$ on nonuniform lattice is given by%
\[
{[x(s)-x(z)]^{(n)}=}%
{\displaystyle\prod\nolimits_{k=0}^{n-1}}
{[x(s)-x(z-k)],(n\in N}^{+}),
\]
and a more formal definition and further properties of the generalized powers
${[x_{\nu}(s)-x_{\nu}(z)]^{(\alpha)}}$ on nonuniform lattice are very
important, which are defined as follows:

\begin{definition}
\textbf{ \label{power}(}See\textbf{ }%
\cite{atakishiyev19882,suslov1992,suslov1989}) Let $\alpha\in C,$ the
generalized powers ${[x_{\nu}(s)-x_{\nu}(z)]^{(\alpha)}}$ are defined by%
\begin{equation}
{[x_{\nu}(s)-x_{\nu}(z)]^{(\alpha)}}=\left\{
\begin{array}
[c]{c}%
\frac{\Gamma(s-z+\alpha)}{\Gamma(s-z)},\text{ \ \ \ \ \ \ if }%
x(s)=s,\text{\ \ \ \ \ \ \ \ \ \ \ \ \ \ \ \ \ \ \ \ \ \ \ \ \ \ \ \ \ \ \ \ \ \ \ \ \ \ \ \ \ \ \ \ }%
\\
\frac{\Gamma(s-z+\alpha)\Gamma(s+z+\nu+1)}{\Gamma(s-z)\Gamma(s+z+\nu
-\alpha+1)},\text{\ if\ }x(s)=s^{2}%
,\text{\ \ \ \ \ \ \ \ \ \ \ \ \ \ \ \ \ \ \ \ \ \ \ \ \ \ \ \ \ \ \ \ \ \ \ \ \ \ }%
\\
(q-1)^{\alpha}q^{\alpha(\nu-\alpha+1)/2}\frac{\Gamma_{q}(s-z+\alpha)}%
{\Gamma_{q}(s-z)},\text{ if }x(s)=q^{s}%
,\text{\ \ \ \ \ \ \ \ \ \ \ \ \ \ \ \ \ \ \ \ \ \ \ \ \ \ \ \ \ \ \ \ \ \ \ \ \ \ \ \ \ \ \ \ }%
\\
\frac{1}{2^{\alpha}}(q-1)^{2\alpha}q^{-\alpha(s+\frac{\nu}{2})}\frac
{\Gamma_{q}(s-z+\alpha)\Gamma_{q}(s+z+\nu+1)}{\Gamma_{q}(s-z)\Gamma
_{q}(s+z+\nu-\alpha+1)},\text{\ if\ }x(s)=\frac{q^{s}+q^{-s}}{2}%
\text{.\ \ \ \ \ \ \ \ \ \ \ \ \ \ \ \ \ \ \ \ \ \ \ \ \ \ \ \ \ \ \ }%
\end{array}
\right.
\end{equation}

For the quadratic lattices of the form (\ref{non1}), with $c=\frac
{\widetilde{c_{2}}}{\widetilde{c_{1}}},$%
\begin{equation}
{[x_{\nu}(s)-x_{\nu}(z)]^{(\alpha)}=}\widetilde{c_{1}}^{\alpha}\frac
{\Gamma(s-z+\alpha)\Gamma(s+z+\nu+c+1)}{\Gamma(s-z)\Gamma(s+z+\nu-\alpha
+c+1)};
\end{equation}

For the $q-$quadratic lattices of the form (\ref{non2}), with $c=\frac
{\log\frac{c_{2}}{c_{1}}}{\log q},$%
\begin{equation}
{[x_{\nu}(s)-x_{\nu}(z)]^{(\alpha)}=[}c_{1}(1-q)^{2}]^{\alpha}q^{-\alpha
(s+\frac{\nu}{2})}\frac{\Gamma_{q}(s-z+\alpha)\Gamma_{q}(s+z+\nu+c+1)}%
{\Gamma_{q}(s-z)\Gamma_{q}(s+z+\nu-\alpha+c+1)},
\end{equation}
where $\Gamma(s)$ is Euler Gamma function and $\Gamma_{q}(s)$ is Euler
$q-$Gamma function which is defined as (\ref{gammaq}).
\end{definition}

\begin{proposition}
\label{properties}\cite{atakishiyev19882,suslov1992,suslov1989}. \textit{For}
$x(s)=c_{1}q^{s}+c_{2}q^{-s}+c_{3}$ \textit{or} $x(s)=\widetilde{c}_{1}%
s^{2}+\widetilde{c}_{2}s+\widetilde{c}_{3}$, \textit{the generalized power }
${[x_{\nu}(s)-x_{\nu}(z)]^{(\alpha)}}$ satisfy the following properties:%
\begin{align}
{[x_{\nu}(s)-x_{\nu}(z)][x_{\nu}(s)-x_{\nu}(z-1)]^{(\mu)}}  &  {=[x_{\nu
}(s)-x_{\nu}(z)]}^{(\mu)}{[x_{\nu}(s)-x_{\nu}(z-\mu)]}\\
&  {=}{[x_{\nu}(s)-x_{\nu}(z)]^{(\mu+1)};}%
\end{align}%
\begin{align}
&  {[x_{\nu-1}(s+1)-x_{\nu-1}(z)]}^{(\mu)}{[x_{\nu-\mu}(s)-x_{\nu-\mu}%
(z)]}\nonumber\\
{=[x_{\nu-\mu}(s+\mu)-x_{\nu-\mu}(z)]}  &  {[x_{\nu-1}(s)-x_{\nu-1}(z)]}%
^{(\mu)}{=}{[x_{\nu}(s)-x_{\nu}(z)]^{(\mu+1)};}%
\end{align}%
\begin{align}
\frac{\Delta_{z}}{\Delta x_{\nu-\mu+1}(z)}{[x_{\nu}(s)-x_{\nu}(z)]}^{(\mu)}
&  =-\frac{\nabla_{s}}{\nabla x_{\nu+1}(s)}{[x_{\nu+1}(s)-x_{\nu+1}(z)]}%
^{(\mu)}\\
&  =-[\mu]_{q}{[x_{\nu}(s)-x_{\nu}(z)]}^{(\mu-1)};
\end{align}%
\begin{align}
\frac{\nabla_{z}}{\nabla x_{\nu-\mu+1}(z)}\{\frac{{1}}{{[x_{\nu}(s)-x_{\nu
}(z)]}^{(\mu)}}\}  &  =-\frac{\Delta_{s}}{\Delta x_{\nu-1}(s)}\{\frac{{1}%
}{{[x_{\nu-1}(s)-x_{\nu-1}(z)]}^{(\mu)}}\}\\
&  =\frac{[\mu]_{q}}{{[x_{\nu}(s)-x_{\nu}(z)]}^{(\mu+1)}}.
\end{align}
where $[\mu]_{q}$ is defined as (\ref{muq}).
\end{proposition}

Now let us first define the integer sum on non-uniform lattices $x_{\gamma
}(s)$ in detail, which is very helpful for us to define fractional sum on
non-uniform lattices $x_{\gamma}(s).$

For $\gamma\in R,$ the $1$-th order sum of $f(z)$ over $\{a+1,a+2,...,z\}$ on
non-uniform lattices $x_{\gamma}(s)$ is defined by%
\begin{equation}
y_{1}(z)=\nabla_{\gamma}^{-1}f(z)=\int_{a+1}^{z}f(s)d_{\nabla}x_{\gamma}(s),
\end{equation}
then by \textbf{Proposion} \ref{integration}, we have%

\begin{equation}
\nabla_{\gamma}^{1}\nabla_{\gamma}^{-1}f(z)=\frac{\nabla y_{1}(z)}{\nabla
x_{\gamma}(z)}=f(z),
\end{equation}
and $2$-th order sum of $f(z)$ over $\{a+1,a+2,...,z\}$ on non-uniform
lattices $x_{\gamma}(s)$ is defined by%

\begin{align}
y_{2}(z)  &  =\nabla_{\gamma}^{-2}f(z)=\nabla_{\gamma+1}^{-1}[\nabla_{\gamma
}^{-1}f(z)]=\int_{a+1}^{z}y_{1}(s)d_{\nabla}x_{\gamma+1}(s)\nonumber\\
&  =\int_{a+1}^{z}d_{\nabla}x_{\gamma+1}(s)\int_{a+1}^{s}f(t)d_{\nabla
}x_{\gamma}(t)\nonumber\\
&  =\int_{a+1}^{z}f(t)d_{\nabla}x_{\gamma}(t)\int_{t}^{z}d_{\nabla}%
x_{\gamma+1}(s)\nonumber\\
&  =\int_{a+1}^{z}[x_{\gamma+1}(z)-x_{\gamma+1}(t-1)]f(s)d_{\nabla}x_{\gamma
}(s).
\end{align}
Meanwhile, we have%
\begin{align}
\nabla_{\gamma+1}^{1}\nabla_{\gamma+1}^{-1}y_{1}\left(  z\right)   &
=\frac{\nabla y_{2}(z)}{\nabla x_{\gamma+1}(z)}=y_{1}(z),\nonumber\\
\nabla_{\gamma}^{2}\nabla_{\gamma}^{-2}f\left(  z\right)   &  =\frac{\nabla
}{\nabla x_{\gamma}(z)}(\frac{\nabla y_{2}(z)}{\nabla x_{\gamma+1}(z)}%
)=\frac{\nabla y_{1}(z)}{\nabla x_{\gamma}(z)}=f(z),
\end{align}
and $3$-th order sum of $f(z)$ over $\{a+1,a+2,...,z\}$ on non-uniform
lattices $x_{\gamma}(s)$ is defined by%

\begin{align}
y_{3}(z)  &  =\nabla_{\gamma}^{-3}f(z)=\nabla_{\gamma+2}^{-1}[\nabla_{\gamma
}^{-2}f(z)]=\int_{a+1}^{z}y_{2}(s)d_{\nabla}x_{\gamma+2}(s)\nonumber\\
&  =\int_{a+1}^{z}d_{\nabla}x_{\gamma+2}(s)\int_{a+1}^{s}[x_{\gamma
+1}(s)-x_{\gamma+1}(t-1)]f(t)d_{\nabla}x_{\gamma}(t)\nonumber\\
&  =\int_{a+1}^{z}f(t)d_{\nabla}x_{\gamma}(t)\int_{t}^{z}[x_{\gamma
+1}(s)-x_{\gamma+1}(t-1)]d_{\nabla}x_{\gamma+2}(s).\nonumber
\end{align}
In View of \textbf{Proposition} \ref{properties} one has%

\begin{equation}
\frac{\nabla}{\nabla x_{\gamma+2}(s)}[x_{\gamma+2}(s)-x_{\gamma+2}%
(t-1)]^{(2)}=[2]_{q}[x_{\gamma+1}(s)-x_{\gamma+1}(t-1)],
\end{equation}
then by the use of \textbf{Proposition} \ref{integration}, we have%

\begin{equation}
\frac{\lbrack x_{\gamma+2}(z)-x_{\gamma+2}(t-1)]^{(2)}}{[2]_{q}}=\int_{t}%
^{z}[x_{\gamma+1}(s)-x_{\gamma+1}(t-1)]d_{\nabla}x_{\gamma+2}(s).
\end{equation}
Therefore, we obtain that $3$-th order sum of $f(z)$ over $\{a+1,a+2,...,z\}$
on non-uniform lattices $x_{\gamma}(s)$ is%

\begin{align}
y_{3}(z)  &  =\nabla_{\gamma}^{-3}f(z)=\nabla_{\gamma+2}^{-1}[\nabla_{\gamma
}^{-2}f(z)]\nonumber\\
&  =\frac{1}{[\Gamma(3)]_{q}}\int_{a+1}^{z}[x_{\gamma+2}(z)-x_{\gamma
+2}(t-1)]^{(2)}f(s)d_{\nabla}x_{\gamma}(s).
\end{align}
Mean while, it is easy to know that
\begin{equation}
\nabla_{\gamma}^{3}\nabla_{\gamma}^{-3}f\left(  z\right)  =\frac{\nabla
}{\nabla x_{\gamma}(z)}(\frac{\nabla}{\nabla x_{\gamma+1}(z)}(\frac{\nabla
y_{3}(z)}{\nabla x_{\gamma+2}(z)}))=f(z).
\end{equation}

More generalaly, by the induction, we can define the $k$-th order sum of
$f(z)$ over $\{a+1,a+2,...,z\}$ on non-uniform lattices $x_{\gamma}(s)$ as%

\begin{align}
y_{k}(z)  &  =\nabla_{\gamma}^{-k}f(z)=\nabla_{\gamma+k-1}^{-1}[\nabla
_{\gamma}^{-(k-1)}f(z)]=\int_{a+1}^{z}y_{k-1}(s)d_{\nabla}x_{\gamma
+k-1}(s)\nonumber\\
&  =\frac{1}{[\Gamma(k)]_{q}}\int_{a+1}^{z}[x_{\gamma+k-1}(z)-x_{\gamma
+k-1}(t-1)]^{(k-1)}f(t)d_{\nabla}x_{\gamma}(t),(k=1,2,...) \label{intsum}%
\end{align}
where%
\[
\lbrack\Gamma(k)]_{q}=%
\genfrac{\{}{.}{0pt}{}{q^{-(k-1)(k-2)}\Gamma_{q}(k),\text{if }x(s)=c_{1}%
q^{s}+c_{2}q^{-s}+c_{3};}{\Gamma(\alpha),\text{if }x(s)=\widetilde{c}_{1}%
s^{2}+\widetilde{c}_{2}s+\widetilde{c}_{3},}%
\]
which satisfies the following%

\[
\lbrack\Gamma(k+1)]_{q}=[k]_{q}[\Gamma(k)]_{q},[\Gamma(2)]_{q}=[1]_{q}%
[\Gamma(1)]_{q}=1.
\]
And then we have%
\begin{equation}
\nabla_{\gamma}^{k}\nabla_{\gamma}^{-k}f\left(  z\right)  =\frac{\nabla
}{\nabla x_{\gamma}(z)}(\frac{\nabla}{\nabla x_{\gamma+1}(z)}...(\frac{\nabla
y_{k}(z)}{\nabla x_{\gamma+k-1}(z)}))=f(z).(k=1,2,...)
\end{equation}

It is noted that the right hand side of (\ref{intsum}) is still meanful when
$k\in C,$ so we can give the definition of fractional sum of $f(z)$ on
non-uniform lattices $x_{\gamma}(s)$ as follows

\begin{definition}
(Fractional sum on non-uniform lattices)\label{maindef1} For any
$\operatorname{Re}\alpha\in R^{+},$ the $\alpha$-th order sum of $f(z)$ over
$\{a+1,a+2,...,z\}$ on non-uniform lattices (\ref{non2}) and (\ref{non1}) is
defined by%
\begin{equation}
\nabla_{\gamma}^{-\alpha}f(z)=\frac{1}{[\Gamma(\alpha)]_{q}}\int_{a+1}%
^{z}[x_{\gamma+\alpha-1}(z)-x_{\gamma+\alpha-1}(t-1)]^{(\alpha-1)}%
f(s)d_{\nabla}x_{\gamma}(s), \label{asum}%
\end{equation}
where
\[
\lbrack\Gamma(\alpha)]_{q}=%
\genfrac{\{}{.}{0pt}{}{q^{-(s-1)(s-2)}\Gamma_{q}(\alpha),\text{if }%
x(s)=c_{1}q^{s}+c_{2}q^{-s}+c_{3};}{\Gamma(\alpha),\text{if }%
x(s)=\widetilde{c}_{1}s^{2}+\widetilde{c}_{2}s+\widetilde{c}_{3},}%
\]
which satisfies the following%
\[
\lbrack\Gamma(\alpha+1)]_{q}=[\alpha]_{q}[\Gamma(\alpha)]_{q}.
\]

\end{definition}

\section{The Analogue of Euler Beta Formula on Non-uniform Lattices}

Euler Beta formula is well known as%

\[
\int_{0}^{1}(1-t)^{\alpha-1}t^{\beta-1}dt=B(\alpha,\beta)=\frac{\Gamma
(\alpha)\Gamma(\beta)}{\Gamma(\alpha+\beta)},(\operatorname{Re}\alpha
>0,\operatorname{Re}\beta>0)
\]
or%

\[
\int_{a}^{z}\frac{(z-t)^{\alpha-1}}{\Gamma(\alpha)}\frac{(t-a)^{\beta-1}%
}{\Gamma(\beta)}dt=\frac{(z-a)^{\alpha+\beta-1}}{\Gamma(\alpha+\beta
)}.(\operatorname{Re}\alpha>0,\operatorname{Re}\beta>0)
\]

In this section, we obtain the analogue Euler Beta formula on non-uniform
lattices. It is very crucial for us to propose several definitions in this
manuscript. And it is also of independent importance.

\begin{theorem}
(Euler Beta formula on non-uniform lattices)\label{eulerbeta} For any
$\alpha,\beta\in C,$ then for non-uniform lattices $x(s)$, we have%
\begin{align}
&  \int_{a+1}^{z}\frac{[x_{\beta}(z)-x_{\beta}(t-1)]^{(\beta-1)}}%
{[\Gamma(\beta)]_{q}}\frac{[x(t)-x(a)]^{(\alpha)}}{[\Gamma(\alpha+1)]_{q}%
}d_{\nabla}x_{1}(t)\nonumber\\
&  =\frac{[x_{\beta}(z)-x_{\beta}(a)]^{(\alpha+\beta)}}{[\Gamma(\alpha
+\beta+1)]_{q}}.
\end{align}

\end{theorem}

The proof of \textbf{Theorem} \ref{eulerbeta} should use some lemmas.

\begin{lemma}
\label{lem12}For any $\alpha,\beta,$we have
\begin{equation}
\lbrack\alpha+\beta]_{q}x(t)-[\alpha]_{q}x_{-\beta}(t)-[\beta]_{q}x_{\alpha
}(t)=const. \label{lem12eq}%
\end{equation}

\end{lemma}

\begin{proof}
If we set $x(t)=\widetilde{c_{1}}t^{2}+\widetilde{c_{2}}t+\widetilde{c_{3}},$
then left hand side of Eq.(\ref{lem12eq}) is%

\begin{align}
LHS  &  =\widetilde{c_{1}}[(\alpha+\beta)t^{2}-\alpha(t-\frac{\beta}{2}%
)^{2}-\beta(t+\frac{\alpha}{2})^{2}]\nonumber\\
&  +\widetilde{c_{2}}[(\alpha+\beta)t-\alpha(t-\frac{\beta}{2})-\beta
(t+\frac{\alpha}{2})]\\
&  =-\frac{\alpha\beta}{4}(\alpha+\beta)\widetilde{c_{1}}=const.
\end{align}
If we set $x(t)=c_{1}q^{t}+c_{2}q^{-t}+c_{3},$ then left hand side of
Eq.(\ref{lem12eq}) is%

\begin{align}
LHS  &  =c_{1}[\frac{q^{\frac{\alpha+\beta}{2}}-q^{-\frac{\alpha+\beta}{2}}%
}{q^{\frac{1}{2}}-q^{-\frac{1}{2}}}q^{t}-\frac{q^{\frac{\alpha}{2}}%
-q^{-\frac{\alpha}{2}}}{q^{\frac{1}{2}}-q^{-\frac{1}{2}}}q^{t-\frac{\beta}{2}%
}-\frac{q^{\frac{\beta}{2}}-q^{-\frac{\beta}{2}}}{q^{\frac{1}{2}}-q^{-\frac
{1}{2}}}q^{t+\frac{\alpha}{2}}]\nonumber\\
&  +c_{2}[\frac{q^{\frac{\alpha+\beta}{2}}-q^{-\frac{\alpha+\beta}{2}}%
}{q^{\frac{1}{2}}-q^{-\frac{1}{2}}}q^{-t}-\frac{q^{\frac{\alpha}{2}}%
-q^{-\frac{\alpha}{2}}}{q^{\frac{1}{2}}-q^{-\frac{1}{2}}}q^{-t+\frac{\beta}%
{2}}-\frac{q^{\frac{\beta}{2}}-q^{-\frac{\beta}{2}}}{q^{\frac{1}{2}}%
-q^{-\frac{1}{2}}}q^{-t-\frac{\alpha}{2}}]\\
&  =0.
\end{align}

\end{proof}

\begin{lemma}
\label{lem13}For any $\alpha,\beta,$we have
\begin{align}
&  [\alpha+1]_{q}[x_{\beta}(z)-x_{\beta}(t-\beta)]-[\beta]_{q}[x_{1-\alpha
}(t+\alpha)-x_{1-\alpha}(a)]\nonumber\\
&  =[\alpha+1]_{q}[x_{\beta}(z)-x_{\beta}(a-\alpha-\beta)]\nonumber\\
&  -[\alpha+\beta+1]_{q}[x(t)-x(a-\alpha)]. \label{lem}%
\end{align}

\end{lemma}

\begin{proof}
(\ref{lem}) is equivalent to%
\begin{align}
&  [\alpha+\beta+1]_{q}x(t)-[\alpha+1]_{q}x_{\beta}(t-\beta)-[\beta
]_{q}x_{1-\alpha}(t+\alpha)\nonumber\\
&  =[\alpha+\beta+1]_{q}x(a-\alpha)-[\alpha+1]_{q}x_{\beta}(a-\alpha
-\beta)-[\beta]_{q}x_{1-\alpha}(a).
\end{align}
Set $\alpha+1=\widetilde{\alpha},$ we only to prove that%
\begin{align}
&  [\widetilde{\alpha}+\beta]_{q}x(t)-[\widetilde{\alpha}]_{q}x_{\beta
}(t-\beta)-[\beta]_{q}x_{2-\widetilde{\alpha}}(t+\widetilde{\alpha
}-1)\nonumber\\
&  =[\widetilde{\alpha}+\beta]_{q}x(a-\widetilde{\alpha}+1)-[\widetilde{\alpha
}]_{q}x_{\beta}(a-\widetilde{\alpha}+1-\beta)-[\beta]_{q}%
x_{2-\widetilde{\alpha}}(a).
\end{align}
That is
\begin{align}
&  [\widetilde{\alpha}+\beta]_{q}x(t)-[\widetilde{\alpha}]_{q}x_{-\beta
}(t)-[\beta]_{q}x_{\widetilde{\alpha}}(t)\nonumber\\
&  =[\widetilde{\alpha}+\beta]_{q}x(a-\widetilde{\alpha}+1)-[\alpha
]_{q}x_{-\beta}(a-\widetilde{\alpha}+1)-[\beta]_{q}x_{\widetilde{\alpha}%
}(a-\widetilde{\alpha}+1). \label{lem2}%
\end{align}
By \textbf{Lemma} \ref{lem12}, Eq. (\ref{lem2}) holds, and then Eq.
(\ref{lem}) holds.
\end{proof}

Using \textbf{Proposition} \ref{properties} and \textbf{Lemma} \ref{lem13},
now it is time for us to prove \textbf{Theorem} \ref{eulerbeta}.

\textbf{Proof of Theorem \ref{eulerbeta}:} Set
\begin{equation}
\rho(t)=[x(t)-x(a)]^{(\alpha)}[x_{\beta}(z)-x_{\beta}(t-1)]^{(\beta-1)},
\end{equation}
and
\begin{equation}
\sigma(t)=[x_{1-\alpha}(t+\alpha)-x_{1-\alpha}(a)][x_{\beta}(z)-x_{\beta}(t)].
\end{equation}
By \textbf{Proposition} \ref{properties}, since%

\begin{equation}
\lbrack x_{1-\alpha}(t+\alpha)-x_{1-\alpha}(a)][x(t)-x(a)]^{(\alpha)}%
=[x_{1}(t)-x_{1}(a)]^{(\alpha+1)}%
\end{equation}
and%
\begin{equation}
\lbrack x_{\beta}(z)-x_{\beta}(t)][x_{\beta}(z)-x_{\beta}(t-1)]^{(\beta
-1)}=[x_{\beta}(z)-x_{\beta}(t)]^{(\beta)}.
\end{equation}
so that we obtain%
\begin{equation}
\sigma(t)\rho(t)=[x_{1}(t)-x_{1}(a)]^{(\alpha+1)}[x_{\beta}(z)-x_{\beta
}(t)]^{(\beta)},
\end{equation}

Making use of%

\[
\nabla_{t}[f(t)g(t)]=g(t-1)\Delta_{t}[f(t)]+f(t)\nabla_{t}[g(t)],
\]
where%

\[
f(t)=[x_{1}(t)-x_{1}(a)]^{(\alpha+1)},g(t)=[x_{\beta}(z)-x_{\beta}%
(t)]^{(\beta)},
\]
let's calculate the$\frac{\nabla_{t}[\sigma(t)\rho(t)]}{\nabla x_{1}(t)}.$

From \textbf{Proposition} \ref{properties}, we have%

\[
\frac{\nabla_{t}}{\nabla x_{1}(t)}\{[x_{1}(t)-x_{1}(a)]^{(\alpha+1)}%
\}=[\alpha+1]_{q}[x(t)-x(a)]^{(\alpha)},
\]
and%
\begin{align*}
&  \frac{\nabla_{t}}{\nabla x_{1}(t)}\{[x_{\beta}(z)-x_{\beta}(t)]^{(\beta
)}\}\\
&  =\frac{\Delta_{t}}{\Delta x_{1}(t-1)}\{[x_{\beta}(z)-x_{\beta
}(t-1)]^{(\beta)}\}\\
&  =-[\beta]_{q}[x_{\beta}(z)-x_{\beta}(t-1)]^{(\beta-1)}.
\end{align*}
These yield%

\begin{align}
&  \frac{\nabla_{t}}{\nabla x_{1}(t)}\{[x_{1}(t)-x_{1}(a)]^{(\alpha
+1)}[x_{\beta}(z)-x_{\beta}(t)]^{(\beta)}\}\nonumber\\
&  =[\alpha+1]_{q}[x(t)-x(a)]^{(\alpha)}[x_{\beta}(z)-x_{\beta}(t-1)]^{(\beta
)}\nonumber\\
&  -[\beta]_{q}[x_{1}(t)-x_{1}(a)]^{(\alpha+1)}[x_{\beta}(z)-x_{\beta
}(t-1)]^{(\beta-1)}\nonumber\\
&  =\{[\alpha+1]_{q}[x_{\beta}(z)-x_{\beta}(t-\beta)]-[\beta]_{q}[x_{1-\alpha
}(t+\alpha)-x_{1-\alpha}(a)]\}\rho(t)\nonumber\\
&  \equiv\tau(t)\rho(t),
\end{align}
where
\begin{equation}
\tau(t)=[\alpha+1]_{q}[x_{\beta}(z)-x_{\beta}(t-\beta)]-[\beta]_{q}%
[x_{1-\alpha}(t+\alpha)-x_{1-\alpha}(a)].
\end{equation}
this is due to%

\[
\lbrack x_{\beta}(z)-x_{\beta}(t-1)]^{(\beta)}=[x_{\beta}(z)-x_{\beta}%
(t-\beta)][x_{\beta}(z)-x_{\beta}(t-1)]^{(\beta-1)}.
\]
Then from \textbf{Lemma} \ref{lem13}, it yields%
\begin{equation}
\tau(t)=[\alpha+1]_{q}[x_{\beta}(z)-x_{\beta}(a-\alpha-\beta)]-[\alpha
+\beta+1]_{q}[x(t)-x(a-\alpha)].
\end{equation}
So that one gets%
\begin{align*}
&  \frac{\nabla_{t}}{\nabla x_{1}(t)}\{[x_{1}(t)-x_{1}(a)]^{(\alpha
+1)}[x_{\beta}(z)-x_{\beta}(t)]^{(\beta)}\}\\
&  =\{[\alpha+1]_{q}[x_{\beta}(z)-x_{\beta}(a-\alpha-\beta)]\\
&  -[\alpha+\beta+1]_{q}[x(t)-x(a-\alpha)]\}\rho(t).
\end{align*}
Or%
\begin{align}
&  \nabla_{t}\{[x_{1}(t)-x_{1}(a)]^{(\alpha+1)}[x_{\beta}(z)-x_{\beta
}(t)]^{(\beta)}\}\nonumber\\
&  =\{[\alpha+1]_{q}[x_{\beta}(z)-x_{\beta}(a-\alpha-\beta)]\nonumber\\
&  -[\alpha+\beta+1]_{q}[x(t)-x(a-\alpha)]\}\nonumber\\
&  \cdot\lbrack x(t)-x(a)]^{(\alpha)}[y_{\beta}(z)-x_{\beta}(t-1)]^{(\beta
-1)}\nabla x_{1}(t).
\end{align}
Summing from $a+1$ to $z$, we have%
\begin{align}
&
{\displaystyle\sum\limits_{t=a+1}^{z}}
\nabla_{t}\{[x_{1}(t)-x_{1}(a)]^{(\alpha+1)}[x_{\beta}(z)-x_{\beta
}(t)]^{(\beta)}\}\nonumber\\
&  =\int_{a+1}^{z}\{[\alpha+1]_{q}[x_{\beta}(z)-x_{\beta}(a-\alpha
-\beta)]\nonumber\\
&  -[\alpha+\beta+1]_{q}[x(t)-x(a-\alpha)]\}\nonumber\\
&  \cdot\lbrack x(t)-x(a)]^{(\alpha)}[x_{\beta}(z)-x_{\beta}(t-1)]^{(\beta
-1)}d_{\nabla}x_{1}(t). \label{imeq}%
\end{align}
Set%
\begin{equation}
I(\alpha)=\int_{a+1}^{z}[x_{\beta}(z)-x_{\beta}(t-1)]^{(\beta-1)}%
[x(t)-x(a)]^{(\alpha)}d_{\nabla}x_{1}(t), \label{ialpha}%
\end{equation}
and%
\begin{equation}
I(\alpha+1)=\int_{a+1}^{z}[x_{\beta}(z)-x_{\beta}(t-1)]^{(\beta-1)}%
[x(t)-x(a)]^{(\alpha+1)}d_{\nabla}x_{1}(t).
\end{equation}
Then from (\ref{imeq}) and by the use of \textbf{Proposition} \ref{properties}%
, one has%

\begin{align*}
&
{\displaystyle\sum\limits_{t=a+1}^{z}}
\nabla_{t}\{[x_{1}(t)-x_{1}(a)]^{(\alpha+1)}[x_{\beta}(z)-x_{\beta
}(t)]^{(\beta)}\}\\
&  =[\alpha+1]_{q}[x_{\beta}(z)-x_{\beta}(a-\alpha-\beta)]\int_{a+1}%
^{z}[x(t)-x(a)]^{(\alpha)}[x_{\beta}(z)-x_{\beta}(t-1)]^{(\beta-1)}d_{\nabla
}x_{1}(t)\\
&  -[\alpha+\beta+1]_{q}\int_{a+1}^{z}[x(t)-x(a-\alpha)][x(t)-x(a)]^{(\alpha
)}[x_{\beta}(z)-x_{\beta}(t-1)]^{(\beta-1)}d_{\nabla}x_{1}(t)\\
&  =[\alpha+1]_{q}[x_{\beta}(z)-x_{\beta}(a-\alpha-\beta)]\int_{a+1}%
^{z}[x(t)-x(a)]^{(\alpha)}[x_{\beta}(z)-x_{\beta}(t-1)]^{(\beta-1)}d_{\nabla
}x_{1}(t)\\
&  -[\alpha+\beta+1]_{q}\int_{a+1}^{z}[x(t)-x(a)]^{(\alpha+1)}[x_{\beta
}(z)-x_{\beta}(t-1)]^{(\beta-1)}d_{\nabla}x_{1}(t)\\
&  =[\alpha+1]_{q}[x_{\beta}(z)-x_{\beta}(a-\alpha-\beta)I(\alpha
)-[\alpha+\beta+1]_{q}I(\alpha+1).
\end{align*}
Since%
\begin{equation}%
{\displaystyle\sum\limits_{t=a+1}^{z}}
\nabla_{t}\{[x_{1}(t)-x_{1}(a)]^{(\alpha+1)}[x_{\beta}(z)-x_{\beta
}(t)]^{(\beta)}\}=0,
\end{equation}
therefore, we have prove that%

\begin{equation}
\frac{I(\alpha+1)}{I(\alpha)}=\frac{[\alpha+1]_{q}}{[\alpha+\beta+1]_{q}%
}[x_{\beta}(z)-x_{\beta}(a-\alpha-\beta)]. \label{imeq2}%
\end{equation}
From (\ref{imeq2}), one has%
\begin{equation}
\frac{I(\alpha+1)}{I(\alpha)}=\frac{\frac{[\Gamma(\alpha+2)]_{q}}%
{[\Gamma(\alpha+\beta+2)]_{q}}[x_{\beta}(z)-x_{\beta}(a)]^{(\alpha+\beta+1)}%
}{\frac{[\Gamma(\alpha+1)]_{q}}{[\Gamma(\alpha+\beta+1)]_{q}}[x_{\beta
}(z)-x_{\beta}(a)]^{(\alpha+\beta)}}.\nonumber
\end{equation}
So that we can set%

\begin{equation}
I(\alpha)=k\frac{[\Gamma(\alpha+1)]_{q}}{[\Gamma(\alpha+\beta+1)]_{q}%
}[x_{\beta}(z)-x_{\beta}(a)]^{(\alpha+\beta)},
\end{equation}
where $k$ is undetermined.

Set $\alpha=0,$ then%

\begin{equation}
I(0)=k\frac{1}{[\Gamma(\beta+1)]_{q}}[x_{\beta}(z)-x_{\beta}(a)]^{(\beta)},
\label{i01}%
\end{equation}
From (\ref{ialpha}), one has%

\begin{align}
I(0)  &  =\int_{a+1}^{z}[x_{\beta}(z)-x_{\beta}(t-1)]^{(\beta-1)}d_{\nabla
}x_{1}(t)\nonumber\\
&  =\frac{1}{[\beta]_{q}}[x_{\beta}(z)-x_{\beta}(a)]^{(\beta)}, \label{i02}%
\end{align}
From (\ref{i01}) and (\ref{i02}), one gets%

\[
k=\frac{[\Gamma(\beta+1)]_{q}}{[\beta]_{q}}=[\Gamma(\beta)]_{q}.
\]
Hence, we obtain that%

\begin{equation}
I(\alpha)=\frac{[\Gamma(\beta)]_{q}[\Gamma(\alpha+1)]_{q}}{[\Gamma
(\alpha+\beta+1)]_{q}}[x_{\beta}(z)-x_{\beta}(a)]^{(\alpha+\beta)},
\end{equation}
and the proof of \textbf{Theorem }\ref{eulerbeta} is completed.

\section{Generalized Abel Equation and Fractional Difference on Non-uniform
Lattices}

The definition of fractional difference of $f(z)$ on non-uniform lattices
$x_{\gamma}(s)$ seems more difficult and complicated. Our idea is to start by
solving the generalized Abel equation on non-uniform lattices. In detail, an
important question is: Let $m-1<\operatorname{Re}\alpha\leq m,f(z)$ over
$\{a+1,a+2,...,z\}$ be a known function, $g(z)$ over $\{a+1,a+2,...,z\}$ be an
unknown function, which satisfies the following generalized Abel equation
\begin{equation}
\nabla_{\gamma}^{-\alpha}g(z)=\int_{a+1}^{z}\frac{[x_{\gamma+\alpha
-1}(z)-x_{\gamma+\alpha-1}(t-1)]^{(\alpha-1)}}{[\Gamma(\alpha)]_{q}%
}g(t)d_{\nabla}x_{\gamma}(t)=f(t), \label{abel}%
\end{equation}
how to solve generalized Abel equation (\ref{abel})?

In order to solve equation (\ref{abel}), we should use the fundamental
analogue of Euler Beta \textbf{Theorem} \ref{eulerbeta} on non-uniform lattices.

\begin{theorem}
(Solution1 for Abel equation)\label{abelsol} Set functions $f(z)$ and $g(z)$
over $\{a+1,a+2,...,z\}$ satisfy%
\[
\nabla_{\gamma}^{-\alpha}g(z)=f(z),0<m-1<\operatorname{Re}\alpha\leqslant m,
\]
then
\begin{equation}
g(z)=\nabla_{\gamma}^{m}\nabla_{\gamma+\alpha}^{-m+\alpha}f(z)
\end{equation}
holds.
\end{theorem}

\begin{proof}
We only need to prove that%

\[
\nabla_{\gamma}^{-m}g(z)=\nabla_{\gamma+\alpha}^{-(m-\alpha)}f(z).
\]
That is
\[
\nabla_{\gamma+\alpha}^{-(m-\alpha)}f(z)=\nabla_{\gamma+\alpha}^{-(m-\alpha
)}\nabla_{\gamma}^{-\alpha}g(z)=\nabla_{\gamma}^{-m}g(z).
\]
In fact, by \textbf{Definition \ref{maindef1},} we have
\begin{align*}
\nabla_{\gamma+\alpha}^{-(m-\alpha)}f(z)  &  =\int_{a+1}^{z}\frac
{[x_{\gamma+m-1}(z)-x_{\gamma+m-1}(t-1)]^{(m-\alpha-1)}}{[\Gamma
(m-\alpha)]_{q}}f(t)d_{\nabla}{x_{\gamma+\alpha}}(t)\\
&  =\int_{a+1}^{z}\frac{[x_{\gamma+m-1}(z)-x_{\gamma+m-1}(t-1)]^{(m-\alpha
-1)}}{[\Gamma(m-\alpha)]_{q}}d_{\nabla}{x_{\gamma+\alpha}}(t)\\
&  \cdot\int_{a+1}^{z}\frac{[x_{\gamma+\alpha-1}(t)-x_{\gamma+\alpha
-1}(s-1)]^{(\alpha-1)}}{[\Gamma(\alpha)]_{q}}g(s)d_{\nabla}{x_{\gamma}}(s)\\
&  =\int_{a+1}^{z}g(s)\nabla{x_{\gamma}}(s)\int_{s}^{z}\frac{[x_{\gamma
+m-1}(z)-x_{\gamma+m-1}(t-1)]^{(m-\alpha-1)}}{[\Gamma(m-\alpha)]_{q}}\\
&  \cdot\frac{\lbrack x_{\gamma+\alpha-1}(t)-x_{\gamma+\alpha-1}%
(s-1)]^{(\alpha-1)}}{[\Gamma(\alpha)]_{q}}d_{\nabla}{x_{\gamma+\alpha}}(t).
\end{align*}
In \textbf{Theorem \ref{eulerbeta}, }replacing $a+1$ with $s;\alpha$ with
$\alpha-1;\beta$ with $m-\alpha,$ and replacing $x(t)$ with $x_{\nu+\alpha
-1}(t),$ then $x_{\beta}(t)$ with $x_{\nu+m-1}(t),$ we can obtain the
following equality%
\begin{align*}
&  \int_{s}^{z}\frac{[x_{\gamma+m-1}(z)-x_{\gamma+m-1}(t-1)]^{(m-\alpha-1)}%
}{[\Gamma(m-\alpha)]_{q}}\frac{[x_{\gamma+\alpha-1}(t)-x_{\gamma+\alpha
-1}(s-1)]^{(\alpha-1)}}{[\Gamma(\alpha)]_{q}}d_{\nabla}{x_{\gamma+\alpha}%
}(t)\\
&  =\frac{[x_{\gamma+m-1}(z)-x_{\gamma+m-1}(s-1)]^{(-m-1)}}{[\Gamma(m)]_{q}},
\end{align*}
therefore, we have%
\[
\nabla_{\gamma+\alpha}^{-(m-\alpha)}f(z)=\int_{a+1}^{z}\frac{[x_{\gamma
+m-1}(z)-x_{\gamma+m-1}(s-1)]^{(-m-1)}}{[\Gamma(m)]_{q}}g(s)d_{\nabla
}{x_{\gamma}}(s)=\nabla_{\gamma}^{-m}g(z),
\]
which yields%

\[
\nabla_{\gamma}^{m}\nabla_{\gamma+\alpha}^{-(m-\alpha)}f(z)=\nabla_{\gamma
}^{m}\nabla_{\gamma}^{-m}g(z)=g(z).
\]

\end{proof}

Inspired by \textbf{Theorem} \ref{abelsol}, This is natural that we give the
$\alpha$-th order $(0<m-1<\operatorname{Re}\alpha\leq m)$ Riemann-Liouvile
difference of $f(z)$ as follows:

\begin{definition}
(Riemann-Liouvile fractional defference1)\label{maindef2} Let $m$ be the
smallest integer exceeding $\operatorname{Re}\alpha$, $\alpha$-th order
Riemann-Liouvile difference of $f(z)$ over $\{a+1,a+2,...,z\}$ on non-uniform
lattices is defined by
\end{definition}

\begin{equation}
\nabla_{\gamma}^{\alpha}f(z)=\nabla_{\gamma}^{m}(\nabla_{\gamma+\alpha
}^{\alpha-m}f(z)). \label{adiff1}%
\end{equation}

Formally, in \textbf{Definition \ref{maindef1}}, if $\alpha$ is replaced by
$-\alpha$, then the RHS of (\ref{asum}) become%

\begin{align}
&  \int_{a+1}^{z}{\frac{{{{[{x_{\gamma-\alpha-1}}(z)-{x_{\gamma-\alpha-1}%
}(t-1)]}^{(-\alpha-1)}}}}{[{\Gamma(-\alpha)]}_{q}}}f(t)d_{\nabla}{x_{\gamma}%
}(t)\nonumber\\
&  =\frac{\nabla}{\nabla{x_{\gamma-\alpha}(t)}}(\frac{\nabla}{\nabla
{x_{\gamma-\alpha+1}(t)}}...\frac{\nabla}{\nabla{x_{\gamma-\alpha+n-1}(t)}%
})\nonumber\\
&  \cdot\int_{a+1}^{z}{\frac{{{{[{x_{\gamma+n-\alpha-1}}(z)-{x_{\gamma
+n-\alpha-1}}(t-1)]}^{(n-\alpha-1)}}}}{[{\Gamma(n-\alpha)]}_{q}}}%
f(t)d_{\nabla}{x_{\gamma}}(t)\\
&  =\nabla_{\gamma-\alpha}^{n}\nabla_{\gamma}^{-n+\alpha}f(z)=\nabla
_{\gamma-\alpha}^{\alpha}f(z). \label{adiff2}%
\end{align}

From (\ref{adiff2}), we can also obtain $\alpha$-th order difference of $f(z)$
as follows

\begin{definition}
(Riemann-Liouvile fractional defference2)\label{maindef3} Let
$\operatorname{Re}\alpha>0,$ $\alpha$-th order Riemann-Liouvile difference of
$f(z)$ over $\{a+1,a+2,...,z\}$ on non-uniform lattices can be defined by%
\begin{equation}
\nabla_{\gamma-\alpha}^{\alpha}f(z)=\int_{a+1}^{z}{\frac{{{{[{x_{\gamma
-\alpha-1}}(z)-{x_{\gamma-\alpha-1}}(t-1)]}^{(-\alpha-1)}}}}{[{\Gamma
(-\alpha)]}_{q}}}f(t)d_{\nabla}{x_{\gamma}}(t). \label{adiff3}%
\end{equation}
Replacing ${x_{\gamma-\alpha}(t)}$ with ${x_{\gamma}}(t),$ Then%
\begin{equation}
\nabla_{\gamma}^{\alpha}f(z)=\int_{a+1}^{z}{\frac{{{{[{x_{\gamma-1}%
}(z)-{x_{\gamma-1}}(t-1)]}^{(-\alpha-1)}}}}{[{\Gamma(-\alpha)]}_{q[}}%
}f(t)d_{\nabla}{x_{\gamma+\alpha}}(t), \label{adiff4}%
\end{equation}
where $\alpha\notin N.$
\end{definition}

\section{Caputo fractional Difference on Non-uniform Lattices}

In this section, we give suitable definition of Caputo fractional difference
on non-uniform lattices.

\begin{theorem}
(Sum by parts formula) \textit{Given two functions} Let $f(s),g(s)$
\textit{with complex variable} $s$, then%
\[
\int_{a+1}^{z}g(s)\nabla_{\gamma}f(s)d_{\nabla}x_{\gamma}%
(s)=f(z)g(z)-f(a)g(a)-\int_{a+1}^{z}f(s-1)\nabla_{\gamma}g(s)d_{\nabla
}x_{\gamma}(s),
\]
where $z,a\in C,$ and $z-a\in N.$
\end{theorem}

\begin{proof}
Make use of \textbf{Proposition} \ref{pro}, one has%
\[
g(s)\nabla_{\gamma}f(s)=\nabla_{\gamma}[f(z)g(z)]-f(s-1)\nabla_{\gamma}g(s),
\]
it yields
\[
g(s)\nabla_{\gamma}f(s)\nabla x_{\gamma}(s)=\nabla_{\gamma}[f(z)g(z)]\nabla
x_{\gamma}(s)-f(s-1)\nabla_{\gamma}g(s)\nabla x_{\gamma}(s).
\]
Summing from $a+1$ to $z$ with vatiable $s,$ then we get%
\begin{align*}
\int_{a+1}^{z}g(s)\nabla_{\gamma}f(s)d_{\nabla}x_{\gamma}(s)  &  =\int%
_{a+1}^{z}\nabla_{\gamma}[f(z)g(z)]\nabla x_{\gamma}(s)-\int_{a+1}%
^{z}f(s-1)\nabla_{\gamma}g(s)d_{\nabla}x_{\gamma}(s)\\
&  =f(z)g(z)-f(a)g(a)-\int_{a+1}^{z}f(s-1)\nabla_{\gamma}g(s)d_{\nabla
}x_{\gamma}(s).
\end{align*}

\end{proof}

The idea of the definition of Caputo fractional difference on non-uniform
lattices is also inspired by the the solution of generalized Abel equation
(\ref{abel}). In section 4, we have obtained that the solution of the
generalized Abel equation%
\[
\nabla_{\gamma}^{-\alpha}g(z)=f(z),0<m-1<\alpha\leqslant m,
\]
is
\begin{equation}
g(z)=\nabla_{\gamma}^{\alpha}f(z)=\nabla_{\gamma}^{m}\nabla_{\gamma+\alpha
}^{-m+\alpha}f(z). \label{RL}%
\end{equation}
Now we will give a new expression of (\ref{RL}) by parts formula. In fact, we
have
\begin{align}
\nabla_{\gamma}^{\alpha}f(z)  &  =\nabla_{\gamma}^{m}\nabla_{\gamma+\alpha
}^{-m+\alpha}f(z)\nonumber\\
&  =\nabla_{\gamma}^{m}\int_{a+1}^{z}\frac{[x_{\gamma+m-1}(z)-x_{\gamma
+m-1}(s-1)]^{(m-\alpha-1)}}{[\Gamma(m-\alpha)]_{q}}f(s)d_{\nabla}%
x_{\gamma+\alpha}(s). \label{RL1}%
\end{align}

In view of the identity%

\begin{align*}
\frac{\nabla_{(s)}[x_{\gamma+m-1}(z)-x_{\gamma+m-1}(s)]^{(m-\alpha)}}{\nabla
x_{\gamma+\alpha}(s)}  &  =\frac{\Delta_{(s)}[x_{\gamma+m-1}(z)-x_{\gamma
+m-1}(s-1)]^{(m-\alpha)}}{\Delta x_{\gamma+\alpha}(s-1)}\\
&  =-[m-\alpha]_{q}[x_{\gamma+m-1}(z)-x_{\gamma+m-1}(s-1)]^{(m-\alpha-1)},
\end{align*}
then the expression
\[
\int_{a+1}^{z}\frac{[x_{\gamma+m-1}(z)-x_{\gamma+m-1}(s-1)]^{(m-\alpha-1)}%
}{[\Gamma(m-\alpha)]_{q}}f(s)d_{\nabla}x_{\gamma+\alpha}(s)
\]
can be written as%
\begin{align*}
&  \int_{a+1}^{z}f(s)\nabla_{(s)}\{\frac{-[x_{\gamma+m-1}(z)-x_{\gamma
+m-1}(s)]^{(m-\alpha)}}{[\Gamma(m-\alpha+1)]_{q}}\}d_{\nabla}s\\
&  =\int_{a+1}^{z}f(s)\nabla_{\gamma+\alpha-1}\{\frac{-[x_{\gamma
+m-1}(z)-x_{\gamma+m-1}(s)]^{(m-\alpha)}}{[\Gamma(m-\alpha+1)]_{q}}%
\}d_{\nabla}x_{\gamma+\alpha-1}(s).
\end{align*}
Summing by parts formula, we get%
\begin{align*}
&  \int_{a+1}^{z}f(s)\nabla_{\gamma+\alpha-1}\{\frac{-[x_{\gamma
+m-1}(z)-x_{\gamma+m-1}(s)]^{(m-\alpha)}}{[\Gamma(m-\alpha+1)]_{q}}%
\}d_{\nabla}x_{\gamma+\alpha-1}(s)\\
&  =f(a)\frac{[x_{\gamma+m-1}(z)-x_{\gamma+m-1}(a)]^{(m-\alpha)}}%
{[\Gamma(m-\alpha+1)]_{q}}\\
&  +\int_{a+1}^{z}\frac{[x_{\gamma+m-1}(z)-x_{\gamma+m-1}(s-1)]^{(m-\alpha)}%
}{[\Gamma(m-\alpha+1)]_{q}}\nabla_{\gamma+\alpha-1}[f(s)]d_{\nabla}%
x_{\gamma+\alpha-1}(s).
\end{align*}
Therefore, it reduce to%
\begin{align}
&  \int_{a+1}^{z}\frac{[x_{\gamma+m-1}(z)-x_{\gamma+m-1}(s-1)]^{(m-\alpha-1)}%
}{[\Gamma(m-\alpha)]_{q}}f(s)d_{\nabla}x_{\gamma+\alpha}(s)\label{s1}\\
&  =f(a)\frac{[x_{\gamma+m-1}(z)-x_{\gamma+m-1}(a)]^{(m-\alpha)}}%
{[\Gamma(m-\alpha+1)]_{q}}\nonumber\\
&  +\int_{a+1}^{z}\frac{[x_{\gamma+m-1}(z)-x_{\gamma+m-1}(s-1)]^{(m-\alpha)}%
}{[\Gamma(m-\alpha+1)]_{q}}\nabla_{\gamma+\alpha-1}[f(s)]d_{\nabla}%
x_{\gamma+\alpha-1}(s).\nonumber
\end{align}
Further, consider
\begin{equation}
\int_{a+1}^{z}\frac{[x_{\gamma+m-1}(z)-x_{\gamma+m-1}(s-1)]^{(m-\alpha)}%
}{[\Gamma(m-\alpha+1)]_{q}}\nabla_{\gamma+\alpha-1}[f(s)]d_{\nabla}%
x_{\gamma+\alpha-1}(s). \label{exp2}%
\end{equation}
By the use of the identity%
\begin{align*}
\frac{\nabla_{(s)}[x_{\gamma+m-1}(z)-x_{\gamma+m-1}(s)]^{(m-\alpha+1)}}{\nabla
x_{\gamma+\alpha-1}(s)}  &  =\frac{\Delta_{(s)}[x_{\gamma+m-1}(z)-x_{\gamma
+m-1}(s-1)]^{(m-\alpha+1)}}{\Delta x_{\gamma+\alpha-1}(s-1)}\\
&  =-[m-\alpha+1]_{q}[x_{\gamma+m-1}(z)-x_{\gamma+m-1}(s-1)]^{(m-\alpha)},
\end{align*}
the expression (\ref{exp2}) can be written as%
\begin{align*}
&  \int_{a+1}^{z}\nabla_{\gamma+\alpha-1}[f(s)]\nabla_{(s)}\{\frac
{-[x_{\gamma+m-1}(z)-x_{\gamma+m-1}(s-1)]^{(m-\alpha+1)}}{[\Gamma
(m-\alpha+2)]_{q}}\}d_{\nabla}s\\
&  =\int_{a+1}^{z}\nabla_{\gamma+\alpha-1}[f(s)]\nabla_{\gamma+\alpha
-2}\{\frac{-[x_{\gamma+m-1}(z)-x_{\gamma+m-1}(s-1)]^{(m-\alpha+1)}}%
{[\Gamma(m-\alpha+2)]_{q}}\}d_{\nabla}x_{\gamma+\alpha-2}(s).
\end{align*}
Summing by parts formula, we have%
\begin{align*}
&  \int_{a+1}^{z}\nabla_{\gamma+\alpha-1}[f(s)]\nabla_{\gamma+\alpha-2}%
\{\frac{-[x_{\gamma+m-1}(z)-x_{\gamma+m-1}(s-1)]^{(m-\alpha+1)}}%
{[\Gamma(m-\alpha+2)]_{q}}\}d_{\nabla}x_{\gamma+\alpha-2}(s)\\
&  =\nabla_{\gamma+\alpha-1}f(a)\frac{[x_{\gamma+m-1}(z)-x_{\gamma
+m-1}(a)]^{(m-\alpha+1)}}{[\Gamma(m-\alpha+2)]_{q}}+\\
&  +\int_{a+1}^{z}\frac{[x_{\gamma+m-1}(z)-x_{\gamma+m-1}(s-1)]^{(m-\alpha
+1)}}{[\Gamma(m-\alpha+2)]_{q}}[\nabla_{\gamma+\alpha-2}\nabla_{\gamma
+\alpha-1}]f(s)d_{\nabla}x_{\gamma+\alpha-2}(s)\\
&  =\nabla_{\gamma+\alpha-1}f(a)\frac{[x_{\gamma+m-1}(z)-x_{\gamma
+m-1}(a)]^{(m-\alpha+1)}}{[\Gamma(m-\alpha+2)]_{q}}+\\
&  +\int_{a+1}^{z}\frac{[x_{\gamma+m-1}(z)-x_{\gamma+m-1}(s-1)]^{(m-\alpha
+1)}}{[\Gamma(m-\alpha+2)]_{q}}\nabla_{\gamma+\alpha-2}^{2}f(s)d_{\nabla
}x_{\gamma+\alpha-2}(s)
\end{align*}
Therefore, we conclude that%
\begin{align}
&  \int_{a+1}^{z}\frac{[x_{\gamma+m-1}(z)-x_{\gamma+m-1}(s-1)]^{(m-\alpha)}%
}{[\Gamma(m-\alpha+1)]_{q}}\nabla_{\gamma+\alpha-1}[f(s)]d_{\nabla}%
x_{\gamma+\alpha-1}(s)\nonumber\\
&  =\nabla_{\gamma+\alpha-1}f(a)\frac{[x_{\gamma+m-1}(z)-x_{\gamma
+m-1}(a)]^{(m-\alpha+1)}}{[\Gamma(m-\alpha+2)]_{q}}+\nonumber\\
&  +\int_{a+1}^{z}\frac{[x_{\gamma+m-1}(z)-x_{\gamma+m-1}(s-1)]^{(m-\alpha
+1)}}{[\Gamma(m-\alpha+2)]_{q}}\nabla_{\gamma+\alpha-2}^{2}f(s)d_{\nabla
}x_{\gamma+\alpha-2}(s) \label{s2}%
\end{align}
In the same way, by mathematical induction we can obtain%
\begin{align}
&  \int_{a+1}^{z}\frac{[x_{\gamma+m-1}(z)-x_{\gamma+m-1}(s-1)]^{(m-\alpha
+k-1)}}{[\Gamma(m-\alpha+k)]_{q}}\nabla_{\gamma+\alpha-k}^{k}[f(s)]d_{\nabla
}x_{\gamma+\alpha-k}(s)\nonumber\\
&  =\nabla_{\gamma+\alpha-k}^{k}f(a)\frac{[x_{\gamma+m-1}(z)-x_{\gamma
+m-1}(a)]^{(m-\alpha+k)}}{[\Gamma(m-\alpha+k+1)]_{q}}+\nonumber\\
&  +\int_{a+1}^{z}\frac{[x_{\gamma+m-1}(z)-x_{\gamma+m-1}(s-1)]^{(m-\alpha
+k)}}{[\Gamma(m-\alpha+k+1)]_{q}}\nabla_{\gamma+\alpha-(k+1)}^{k+1}%
f(s)d_{\nabla}x_{\gamma+\alpha-(k+1)}(s).\label{s3}\\
(k  &  =0,1,...,m-1)\nonumber
\end{align}
Substituting (\ref{s1}), (\ref{s2}) and (\ref{s3}) into (\ref{RL1}), we get%
\begin{align*}
\nabla_{\gamma}^{\alpha}f(z)  &  =\nabla_{\gamma}^{m}\{f(a)\frac
{[x_{\gamma+m-1}(z)-x_{\gamma+m-1}(a)]^{(m-\alpha)}}{[\Gamma(m-\alpha+1)]_{q}%
}+\\
&  +\nabla_{\gamma+\alpha-1}f(a)\frac{[x_{\gamma+m-1}(z)-x_{\gamma
+m-1}(a)]^{(m-\alpha+1)}}{[\Gamma(m-\alpha+2)]_{q}}+\\
&  +\nabla_{\gamma+\alpha-k}^{k}f(a)\frac{[x_{\gamma+m-1}(z)-x_{\gamma
+m-1}(a)]^{[m-\alpha+k]}}{[\Gamma(m-\alpha+k+1)]_{q}}\\
&  +...+\nabla_{\gamma+\alpha-(m-1)}^{m-1}f(a)\frac{[x_{\gamma+m-1}%
(z)-x_{\gamma+m-1}(a)]^{(2m-\alpha-1)}}{[\Gamma(2m-\alpha)]_{q}}+\\
&  +\int_{a+1}^{z}\frac{[x_{\gamma+m-1}(z)-x_{\gamma+m-1}(s-1)]^{(2m-\alpha
-1)}}{[\Gamma(2m-\alpha)]_{q}}\nabla_{\gamma+\alpha-m}^{m}f(s)d_{\nabla
}x_{\gamma+\alpha-m}(s)\}\\
&  =\nabla_{\gamma}^{m}\{\sum_{k=0}^{m-1}\nabla_{\gamma+\alpha-k}^{k}%
f(a)\frac{[x_{\gamma+m-1}(z)-x_{\gamma+m-1}(a)]^{(m-\alpha+k)}}{[\Gamma
(m-\alpha+k+1)]_{q}}+\\
&  +\nabla_{\gamma+\alpha-m}^{\alpha-2m}\nabla_{\gamma+\alpha-m}^{m}f(z)\}\\
&  =\sum_{k=0}^{m-1}\nabla_{\gamma+\alpha-k}^{k}f(a)\frac{[x_{\gamma
-1}(z)-x_{\gamma-1}(a)]^{(-\alpha+k)}}{[\Gamma(-\alpha+k+1)]_{q}}%
+\nabla_{\gamma+\alpha-m}^{\alpha-m}\nabla_{\gamma+\alpha-m}^{m}f(z).
\end{align*}

As a result, we have the following

\begin{theorem}
(Solution2 for Abel equation)\label{abelsol2} Set functions $f(z)$ and $g(z)$
over $\{a+1,a+2,...,z\}$ satisfy%
\[
\nabla_{\gamma}^{-\alpha}g(z)=f(z),0<m-1<\operatorname{Re}\alpha\leqslant m,
\]
then
\begin{equation}
g(z)=\sum_{k=0}^{m-1}\nabla_{\gamma+\alpha-k}^{k}f(a)\frac{[x_{\gamma
-1}(z)-x_{\gamma-1}(a)]^{(-\alpha+k)}}{[\Gamma(-\alpha+k+1)]_{q}}%
+\nabla_{\gamma+\alpha-m}^{\alpha-m}\nabla_{\gamma+\alpha-m}^{m}f(z)
\label{c1}%
\end{equation}
holds.
\end{theorem}

Inspired by \textbf{Theorem} \ref{abelsol2}, This is also natural that we give
the $\alpha$-th order $(0<m<\operatorname{Re}\alpha\leq m-1)$ Caputo
fractional difference of $f(z)$ as follows:

\begin{definition}
(Caputo fractional difference)\label{caputodef}Let $m$ be the smallest integer
exceeding $\operatorname{Re}\alpha$, $\alpha$-th order Caputo fractional
difference of $f(z)$ over $\{a+1,a+2,...,z\}$ on non-uniform lattices is
defined by%
\begin{equation}
^{C}\nabla_{\gamma}^{\alpha}f(z)=\nabla_{\gamma+\alpha-m}^{\alpha-m}%
\nabla_{\gamma+\alpha-m}^{m}f(z).
\end{equation}

\end{definition}

\section{Some Propositions and Theorems}

Some fundamental Propositions, Theorems and Taylor fomula on non-uniform
lattices are very important. We will take effort to establish it in this
section. First, It is easy to prove that

\begin{lemma}
\label{lem22}Let $\alpha>0,$ then
\begin{equation}
\nabla_{\gamma}^{-\alpha}1=\frac{[x_{\gamma+\alpha-1}(z)-x_{\gamma+\alpha
-1}(a)]^{(\alpha)}}{[\Gamma(\alpha+1)]_{q}}.\nonumber
\end{equation}

\end{lemma}

\begin{proof}
By the use of \textbf{Propoition} \ref{properties}, one has%
\begin{equation}
\frac{\nabla_{t}[x_{\gamma+\alpha-1}(z)-x_{\gamma+\alpha-1}(t)]^{(\alpha)}%
}{\nabla x_{\gamma}(t)}=-[\alpha]_{q}[x_{\gamma+\alpha-1}(z)-x_{\gamma
+\alpha-1}(t-1)]^{(\alpha-1)}.
\end{equation}
It is easy to know that%
\begin{align}
\nabla_{\gamma}^{-\alpha}1  &  =\sum_{t=a+1}^{z}\frac{[x_{\gamma+\alpha
-1}(z)-x_{\gamma+\alpha-1}(t-1)]^{(\alpha-1)}}{[\Gamma(\alpha)]_{q}}\nabla
x_{\gamma}(t)\nonumber\\
&  =-\sum_{t=a+1}^{z}\frac{\nabla_{t}[x_{\gamma+\alpha-1}(z)-x_{\gamma
+\alpha-1}(t)]^{(\alpha)}}{[\Gamma(\alpha)]_{q}}\nonumber\\
&  =\frac{[x_{\gamma+\alpha-1}(z)-x_{\gamma+\alpha-1}(a)]^{(\alpha)}}%
{[\Gamma(\alpha+1)]_{q}}.
\end{align}

\end{proof}

\begin{theorem}
\label{taylork}(Taylor Theorem) Let $k\in N$, then
\begin{align}
\nabla_{\gamma}^{-k}\nabla_{\gamma}^{k}f(z)  &  =f(z)-f(a)-\nabla_{\gamma
+k-1}^{1}f(a)[x_{\gamma+k-1}(z)-x_{\gamma+k-1}(a)]\nonumber\\
&  -\frac{1}{[2]_{q}!}\nabla_{\gamma+k-2}^{2}f(a)[x_{\gamma+k-1}%
(z)-x_{\gamma+k-1}(a)]^{(2)}\nonumber\\
&  -...-\frac{1}{[k-1]_{q}!}\nabla_{\gamma+1}^{k-1}f(a)[x_{\gamma
+k-1}(z)-x_{\gamma+k-1}(a)]^{(k-1)}\nonumber\\
&  =f(z)-\sum_{j=0}^{k-1}\frac{1}{[j]_{q}!}\nabla_{\gamma+k-j}^{j}%
f(a)[x_{\gamma+k-1}(z)-x_{\gamma+k-1}(a)]^{(j)}. \label{taylor}%
\end{align}

\end{theorem}

\begin{proof}
When $k=1,$ we should prove that
\begin{equation}
\nabla_{\gamma}^{-1}\nabla_{\gamma}^{1}f(z)=f(z)-f(a). \label{taylor1}%
\end{equation}
In fact, one has%
\[
LHS=%
{\displaystyle\sum\limits_{s=a+1}^{z}}
\nabla_{\gamma}^{1}f(s)\nabla x_{\gamma}(s)=%
{\displaystyle\sum\limits_{s=a+1}^{z}}
\nabla f(s)=f(z)-f(a).
\]

When $k=2,$ we should prove that
\begin{equation}
\nabla_{\gamma}^{-2}\nabla_{\gamma}^{2}f(z)=f(z)-f(a)-\nabla_{\gamma+1}%
^{1}f(a)[x_{\gamma+1}(z)-x_{\gamma+1}(a)]. \label{taylor2}%
\end{equation}
Actually, we have%
\[
\nabla_{\gamma}^{-2}\nabla_{\gamma}^{2}f(z)=\nabla_{\gamma+1}^{-1}%
\nabla_{\gamma}^{-1}\nabla_{\gamma}^{1}\nabla_{\gamma+1}^{1}f(z)=\nabla
_{\gamma+1}^{-1}[\nabla_{\gamma}^{-1}\nabla_{\gamma}^{1}]\nabla_{\gamma+1}%
^{1}f(z),
\]
by the use of (\ref{taylor1}) and \textbf{Lemma} \ref{lem22}, we have%
\begin{align}
\nabla_{\gamma+1}^{-1}[\nabla_{\gamma}^{-1}\nabla_{\gamma}^{1}]\nabla
_{\gamma+1}^{1}f(z)  &  =\nabla_{\gamma+1}^{-1}[\nabla_{\gamma+1}%
^{1}f(z)-\nabla_{\gamma+1}^{1}f(a)]\nonumber\\
&  =f(z)-f(a)-\nabla_{\gamma+1}^{-1}[\nabla_{\gamma+1}^{1}f(a)]\nonumber\\
&  =f(z)-f(a)-\nabla_{\gamma+1}^{1}f(a)[x_{\gamma+1}(z)-x_{\gamma+1}(a)].
\end{align}
Assume that when $n=k,$ (\ref{taylor}) holds, then for $n=k+1,$ we should
prove that%
\begin{equation}
\nabla_{\gamma}^{-(k+1)}\nabla_{\gamma}^{k+1}f(z)=f(z)-\sum_{j=0}^{k}\frac
{1}{[j]_{q}!}\nabla_{\gamma+k-j+1}^{j}f(a)[x_{\gamma+k}(z)-x_{\gamma
+k}(a)]^{(j)}.
\end{equation}
In fact, we have%
\begin{align}
\nabla_{\gamma}^{-(k+1)}\nabla_{\gamma}^{k+1}f(z)  &  =\nabla_{\gamma+k}%
^{-1}\nabla_{\gamma}^{-k}\nabla_{\gamma}^{k}\nabla_{\gamma+k}^{1}%
f(z)=\nabla_{\gamma+k}^{-1}[\nabla_{\gamma}^{-k}\nabla_{\gamma}^{k}%
]\nabla_{\gamma+k}^{1}f(z)\nonumber\\
&  =\nabla_{\gamma+k}^{-1}\{\nabla_{\gamma+k}^{1}f(z)-\sum_{j=0}^{k-1}\frac
{1}{[j]_{q}!}\nabla_{\gamma+k-j}^{j}\nabla_{\gamma+k}^{1}f(a)[x_{\gamma
+k-1}(z)-x_{\gamma+k-1}(a)]^{(j)}\}\nonumber\\
&  =f(z)-f(a)-\sum_{j=0}^{k-1}\frac{1}{[j+1]_{q}!}\nabla_{\gamma+k-j}%
^{j}\nabla_{\gamma+k}^{1}f(a)[x_{\gamma+k-1}(z)-x_{\gamma+k-1}(a)]^{(j+1)}\}\\
&  =f(z)-\sum_{j=0}^{k}\frac{1}{[j]_{q}!}\nabla_{\gamma+k-j+1}^{j}%
f(a)[x_{\gamma+k}(z)-x_{\gamma+k}(a)]^{(j)},
\end{align}
the last equation holds is duo to%

\begin{align}
\frac{\nabla}{\nabla x_{\gamma+k}(z)}\nabla_{\gamma+k}^{-1}[x_{\gamma
+k-1}(z)-x_{\gamma+k-1}(a)]^{(j)}  &  =[x_{\gamma+k-1}(z)-x_{\gamma
+k-1}(a)]^{(j)}\nonumber\\
&  =\frac{1}{[j+1]_{q}}\frac{\nabla}{\nabla x_{\gamma+k}(z)}[x_{\gamma
+k}(z)-x_{\gamma+k}(a)]^{(j+1)},
\end{align}
hence it holds that%

\begin{equation}
\nabla_{\gamma+k}^{-1}[x_{\gamma+k-1}(z)-x_{\gamma+k-1}(a)]^{(j)}=\frac
{1}{[j+1]_{q}}[x_{\gamma+k}(z)-x_{\gamma+k}(a)]^{(j+1)}.
\end{equation}

\end{proof}

\begin{proposition}
\label{banqun}For any $\operatorname{Re}\alpha,\operatorname{Re}\beta>0,$ we have%

\begin{equation}
\nabla_{\gamma+\alpha}^{-\beta}\nabla_{\gamma}^{-\alpha}f(z)=\nabla
_{\gamma+\beta}^{-\alpha}\nabla_{\gamma}^{-\beta}f(z)=\nabla_{\gamma
}^{-(\alpha+\beta)}f(z). \label{Pro1}%
\end{equation}

\end{proposition}

\begin{proof}
By \textbf{Definition} \ref{maindef1}, we have
\begin{align*}
\nabla_{\gamma+\alpha}^{-\beta}\nabla_{\gamma}^{-\alpha}f(z)  &  =\sum
_{t=a+1}^{z}\frac{[x_{\gamma+\alpha+\beta-1}(z)-x_{\gamma+\alpha+\beta
-1}(t-1)]^{(\beta-1)}}{[\Gamma(\beta)]_{q}}\nabla_{\gamma}^{-\alpha}f(t)\nabla
x_{\gamma+\alpha}(t)\\
&  =\sum_{t=a+1}^{z}\frac{[x_{\gamma+\alpha+\beta-1}(z)-x_{\gamma+\alpha
+\beta-1}(t-1)]^{(\beta-1)}}{[\Gamma(\beta)]_{q}}\nabla x_{\gamma+\alpha}(t)\\
&  \sum_{s=a+1}^{t}\frac{[x_{\gamma+\alpha-1}(t)-x_{\gamma+\alpha
-1}(s-1)]^{(\alpha-1)}}{[\Gamma(\alpha)]_{q}}f(s)\nabla x_{\gamma}(s)\\
&  =\sum_{s=a+1}^{z}f(s)\nabla x_{\gamma}(s)\sum_{t=s}^{z}\frac{[x_{\gamma
+\alpha+\beta-1}(z)-x_{\gamma+\alpha+\beta-1}(t-1)]^{(\beta-1)}}{[\Gamma
(\beta)]_{q}}\\
&  \frac{[x_{\gamma+\alpha-1}(t)-x_{\gamma+\alpha-1}(s-1)]^{(\alpha-1)}%
}{[\Gamma(\alpha)]_{q}}\nabla x_{\gamma+\alpha}(t).
\end{align*}
In \textbf{Theorem \ref{eulerbeta}, }replacing $a+1$ with $s;\alpha$ with
$\alpha-1;$ and replacing $x(t)$ with $x_{\nu+\alpha-1}(t),$ then $x_{\beta
}(t)$ with $x_{\nu+\alpha+\beta-1}(t),$ we get that
\begin{align*}
&  \sum_{t=s}^{z}\frac{[x_{\gamma+\alpha+\beta-1}(z)-x_{\gamma+\alpha+\beta
-1}(t-1)]^{(\beta-1)}}{[\Gamma(\beta)]_{q}}\\
&  \cdot\frac{\lbrack x_{\gamma+\alpha-1}(t)-x_{\gamma+\alpha-1}%
(s-1)]^{(\alpha-1)}}{[\Gamma(\alpha)]_{q}}\nabla x_{\gamma+\alpha}(t)\\
&  =\frac{[x_{\gamma+\alpha+\beta-1}(z)-x_{\gamma+\alpha+\beta-1}%
(s-1)]^{(\alpha+\beta-1)}}{[\Gamma(\alpha+\beta)]_{q}}%
\end{align*}
it yields%

\begin{align*}
\nabla_{\gamma+\alpha}^{-\beta}\nabla_{\gamma}^{-\alpha}f(z)  &  =\sum
_{s=a+1}^{z}\frac{[x_{\gamma+\alpha+\beta-1}(z)-x_{\gamma+\alpha+\beta
-1}(s-1)]^{(\alpha+\beta-1)}}{[\Gamma(\alpha+\beta)]_{q}}f(s)\nabla x_{\gamma
}(s)\\
&  =\nabla_{\gamma}^{-(\alpha+\beta)}f(z).
\end{align*}

\end{proof}

\begin{proposition}
\label{basic}For any $\operatorname{Re}\alpha>0,$ we have%
\begin{equation}
\nabla_{\gamma}^{\alpha}\nabla_{\gamma}^{-\alpha}f(z)=f(z).
\end{equation}

\end{proposition}

\begin{proof}
By \textbf{Definition \ref{maindef1}}, we have%
\begin{equation}
\nabla_{\gamma}^{\alpha}\nabla_{\gamma}^{-\alpha}f(z)=\nabla_{\gamma}%
^{m}(\nabla_{\gamma+\alpha}^{\alpha-m})\nabla_{\gamma}^{-\alpha}f(z).
\end{equation}
In view of \textbf{Proposition \ref{banqun}, }one gets%
\[
\nabla_{\gamma+\alpha}^{\alpha-m}\nabla_{\gamma}^{-\alpha}f(z)=\nabla_{\gamma
}^{-m}f(z).
\]
Therefore, we obtain%
\[
\nabla_{\gamma}^{\alpha}\nabla_{\gamma}^{-\alpha}f(z)=\nabla_{\gamma}%
^{m}\nabla_{\gamma}^{-m}f(z)=f(z).
\]

\end{proof}

\begin{proposition}
\label{lem26}Let $m\in N^{+},\alpha>0,$ then%
\begin{equation}
\nabla_{\gamma}^{m}\nabla_{\gamma+m-\alpha}^{-\alpha}f(z)=%
\genfrac{\{}{.}{0pt}{}{\nabla_{\gamma+m-\alpha}^{m-\alpha}f(z),(\text{when
}m-\alpha<0)}{\nabla_{\gamma}^{m-\alpha}f(z).(\text{when }m-\alpha>0)}
\label{pro2}%
\end{equation}

\end{proposition}

\begin{proof}
If $0\leq\alpha<1,$ set $\beta=m-\alpha,$ then $0\leq m-1<\beta\leq m.$ By
\textbf{Definition} \ref{maindef2}, one has%
\[
\nabla_{\gamma}^{\beta}f(z)=\nabla_{\gamma}^{m}\nabla_{\gamma+\beta}^{\beta
-m}f(z),
\]
that is%
\begin{equation}
\nabla_{\gamma}^{m}\nabla_{\gamma+m-\alpha}^{-\alpha}f(z)=\nabla_{\gamma
}^{m-\alpha}f(z). \label{1}%
\end{equation}
If $k\leq\alpha<k+1,k\in N^{+},$set $\widetilde{\alpha}=\alpha-k,$ then
$0\leq\widetilde{\alpha}<1,$ one has%

\[
\nabla_{\gamma}^{m}\nabla_{\gamma+m-\alpha}^{-\alpha}f(z)=\nabla_{\gamma}%
^{m}\nabla_{\gamma+m-k-\widetilde{\alpha}}^{-k-\widetilde{\alpha}}%
f(z)=\nabla_{\gamma}^{m}\nabla_{\gamma+m-k}^{-k}\nabla_{\gamma
+m-k-\widetilde{\alpha}}^{-\widetilde{\alpha}}f(z)
\]
When $m-k>0,$ we have%

\[
\nabla_{\gamma}^{m}\nabla_{\gamma+m-k}^{-k}\nabla_{\gamma
+m-k-\widetilde{\alpha}}^{-\widetilde{\alpha}}f(z)=\nabla_{\gamma}^{m-k}%
\nabla_{\gamma+m-k-\widetilde{\alpha}}^{-\widetilde{\alpha}}f(z),
\]
since $m-k-\widetilde{\alpha}=m-\alpha>0,$ From (\ref{1}), one has%

\[
\nabla_{\gamma}^{m-k}\nabla_{\gamma+m-k-\widetilde{\alpha}}%
^{-\widetilde{\alpha}}f(z)=\nabla_{\gamma}^{m-k-\widetilde{\alpha}}%
f(z)=\nabla_{\gamma}^{m-\alpha}f(z).
\]
When $m-k<0,$ we have%
\[
\nabla_{\gamma}^{m}\nabla_{\gamma+m-k}^{-k}\nabla_{\gamma
+m-k-\widetilde{\alpha}}^{-\widetilde{\alpha}}f(z)=\nabla_{\gamma+m-k}%
^{m-k}\nabla_{\gamma+m-k-\widetilde{\alpha}}^{-\widetilde{\alpha}}f(z),
\]
since $m-k-\widetilde{\alpha}=m-\alpha<0,$ From (\ref{Pro1}), we obtain%
\[
\nabla_{\gamma+m-k}^{m-k}\nabla_{\gamma+m-k-\widetilde{\alpha}}%
^{-\widetilde{\alpha}}f(z)=\nabla_{\gamma+m-k}^{m-k-\widetilde{\alpha}%
}f(z)=\nabla_{\gamma+m-k}^{m-\alpha}f(z).
\]
Obviously, $m-k>0$ or $m-k>0$ is equivalent to $m-\alpha>0$ or $m-\alpha>0,$
hence it yields%
\[
\nabla_{\gamma}^{m}\nabla_{\gamma+m-\alpha}^{-\alpha}f(z)=%
\genfrac{\{}{.}{0pt}{}{\nabla_{\gamma+m-\alpha}^{m-\alpha}f(z),(\text{when
}m-\alpha<0)}{\nabla_{\gamma}^{m-\alpha}f(z).(\text{when }m-\alpha>0)}%
\]

\end{proof}

\begin{proposition}
\label{lem27}Let $\alpha>0,\beta>0$, We have%
\[
\nabla_{\gamma}^{\beta}\nabla_{\gamma+\beta-\alpha}^{-\alpha}f(z)=%
\genfrac{\{}{.}{0pt}{}{\nabla_{\gamma+\beta-\alpha}^{\beta-\alpha
}f(z),(\text{when }\beta-\alpha<0)}{\nabla_{\gamma}^{\beta-\alpha
}f(z).(\text{when }\beta-\alpha>0)}%
\]

\end{proposition}

\begin{proof}
Let $m$ be the smallest integer exceeding $\beta,$ then by \textbf{Definition}
\ref{maindef2}, we have%

\begin{align*}
\nabla_{\gamma}^{\beta}\nabla_{\gamma+\beta-\alpha}^{-\alpha}f(z)  &
=\nabla_{\gamma}^{m}\nabla_{\gamma+\beta}^{\beta-m}\nabla_{\gamma+\beta
-\alpha}^{-\alpha}f(z)\\
&  =\nabla_{\gamma}^{m}\nabla_{\gamma+\beta-\alpha}^{\beta-\alpha-m}f(z).
\end{align*}
From \textbf{Proposition} \ref{lem26}, we obtain%

\[
\nabla_{\gamma}^{m}\nabla_{\gamma+\beta-\alpha}^{\beta-\alpha-m}f(z)=%
\genfrac{\{}{.}{0pt}{}{\nabla_{\gamma+\beta-\alpha}^{\beta-\alpha
}f(z),(\text{when }\beta-\alpha<0)}{\nabla_{\gamma}^{\beta-\alpha
}f(z).(\text{when }\beta-\alpha>0)}%
\]

\end{proof}

\begin{proposition}
\label{ftaylor1}(Fractional Taylor formula) Let $\alpha>0,$ $k$ be the
smallest integer exceeding $\alpha,$ then%
\[
\nabla_{\gamma}^{-\alpha}\nabla_{\gamma}^{\alpha}f(z)=f(z)-\sum_{j=0}%
^{k-1}\nabla_{\gamma+k-j}^{j-k+\alpha}f(a)\frac{[x_{\gamma+\alpha
-1}(z)-x_{\gamma+\alpha-1}(a)]^{(\alpha-k+j)}}{[\Gamma(\alpha-k+j+1)]_{q}}.
\]

\end{proposition}

\begin{proof}
By \textbf{Definition} \ref{banqun}, \textbf{Definition} \ref{maindef2} and
\textbf{Proposition} \ref{taylork}, we obtain%

\begin{align*}
\nabla_{\gamma}^{-\alpha}\nabla_{\gamma}^{\alpha}f(z)  &  =\nabla
_{\gamma+\alpha}^{-\alpha+k}\nabla_{\gamma}^{-k}\nabla_{\gamma}^{k}%
\nabla_{\gamma+\alpha}^{\alpha-k}f(z)\\
&  =\nabla_{\gamma+\alpha}^{-\alpha+k}\{\nabla_{\gamma+\alpha}^{\alpha
-k}f(z)-\frac{1}{[j]_{q}!}\sum_{j=0}^{k-1}\nabla_{\gamma+k-j+1}^{j}%
\nabla_{\gamma+\alpha}^{-k+\alpha}f(a)[x_{\gamma+k-1}(z)-x_{\gamma
+k-1}(a)]^{(j)}\},
\end{align*}
From Lemma \ref{lem26}, the use of%
\[
\nabla_{\gamma+k-j+1}^{j}\nabla_{\gamma+\alpha}^{-k+\alpha}f(a)=%
\genfrac{\{}{.}{0pt}{}{\nabla_{\gamma+\alpha}^{-k+\alpha}f(a),(\text{when
}j=0)}{\nabla_{\gamma+k-j+1}^{j-k+\alpha}f(a).(\text{when }j>0)}%
\]
and%
\begin{align*}
\nabla_{\gamma+\alpha}^{-\alpha+k}\{\frac{[x_{\gamma+k-1}(z)-x_{\gamma
+k-1}(a)]^{(j)}}{[\Gamma(j+1)]_{q}}\}  &  =\nabla_{\gamma+\alpha}^{-\alpha
+k}\nabla_{\gamma+k-j}^{-j}(1)=\nabla_{\gamma+k-j}^{-\alpha+k-j}(1)\\
&  =\frac{[x_{\gamma+\alpha-1}(z)-x_{\gamma+\alpha-1}(a)]^{(\alpha-k+j)}%
}{[\Gamma(\alpha-k+j+1)]_{q}},
\end{align*}
reducde to%
\begin{equation}
\nabla_{\gamma}^{-\alpha}\nabla_{\gamma}^{\alpha}f(z)=f(z)-\sum_{j=0}%
^{k-1}\nabla_{\gamma+k-j+1}^{j-k+\alpha}\frac{f(a)[x_{\gamma+\alpha
-1}(z)-x_{\gamma+\alpha-1}(a)]^{(\alpha-k+j)}}{[\Gamma(\alpha-k+j+1)]_{q}}.
\end{equation}

\end{proof}

\begin{theorem}
\label{ftaylor2}(Caputo type Fractional Taylor formula) Let $0<k-1<\alpha\leq
k,$ then%
\begin{equation}
\nabla_{\gamma}^{-\alpha}[^{C}\nabla_{\gamma}^{\alpha}]f(z)=f(t)-\sum
_{j=0}^{k-1}(_{a}\nabla_{\gamma+(\alpha-j)}^{k})f(a)\frac{[x_{\gamma
+\alpha-(j+1)}(z)-x_{\gamma+\alpha-(j+1)}(a)]^{(j)}}{[\Gamma(j+1)]_{q}}.\
\end{equation}

\end{theorem}

\begin{proof}
By \textbf{Definition} \ref{caputodef} and \textbf{Proposition} \ref{taylork},
we have
\begin{align}
\nabla_{\gamma}^{-\alpha}[^{C}\nabla_{\gamma}^{\alpha}]f(z)  &  =\nabla
_{\gamma}^{-\alpha}\nabla_{\gamma+(\alpha-k)}^{\alpha-k}\nabla_{\gamma
+(\alpha-k)}^{k}f(z)\nonumber\\
&  =\nabla_{\gamma+(\alpha-k)}^{-k}\nabla_{\gamma+(\alpha-k)}^{k}%
f(z)\nonumber\\
&  =f(t)-\sum_{j=0}^{k-1}(_{a}\nabla_{\gamma+(\alpha-j)}^{j})f(a)\frac
{[x_{\gamma+\alpha-(j+1)}(z)-x_{\gamma+\alpha-(j+1)}(a)]^{(j)}}{[\Gamma
(j+1)]_{q}}\
\end{align}

\end{proof}

The relationship between Riemann-Liouville fractional difference and Caputo
fractional difference is

\begin{proposition}
\label{pro20}Let $m$ be the smallest integer exceeding $\alpha$, we have
\begin{equation}
_{a}^{C}\nabla_{\gamma}^{\alpha}f(z)=[_{a}\nabla_{\gamma}^{\alpha}%
]\{f(t)-\sum_{k=0}^{m-1}(_{a}\nabla_{\gamma+(\alpha-k)}^{k})f(a)\frac
{[x_{\gamma+\alpha-(k+1)}(z)-x_{\gamma+\alpha-(k+1)}(a)]^{(k)}}{[\Gamma
(k+1)]_{q}}\}.\nonumber
\end{equation}

\end{proposition}

\begin{proof}
We have%
\begin{align}
_{a}^{C}\nabla_{\gamma}^{\alpha}f(z)  &  =[_{a}\nabla_{\gamma+(\alpha
-m)}^{\alpha-m}(_{a}\nabla_{\gamma+(\alpha-m)}^{m})]f(z)=[(_{a}\nabla_{\gamma
}^{\alpha})(\nabla_{\gamma+(\alpha-m)}^{-m})(_{a}\nabla_{\gamma+(\alpha
-m)}^{m})]f(z)\nonumber\\
&  =[_{a}\nabla_{\gamma}^{\alpha}]\{f(t)-\sum_{k=0}^{m-1}(_{a}\nabla
_{\gamma+(\alpha-k)}^{k})f(a)\frac{[x_{\gamma+\alpha-(k+1)}(z)-x_{\gamma
+\alpha-(k+1)}(a)]^{(k)}}{[\Gamma(k+1)]_{q}}\}.
\end{align}

\end{proof}

\begin{proposition}
\label{pro21}Let $\alpha>0$, we have%
\begin{equation}
(_{a}^{C}\nabla_{\gamma}^{\alpha})(_{a}\nabla_{\gamma}^{-\alpha})f(z)=f(z).
\end{equation}

\end{proposition}

\begin{proof}
Set
\[
g(z)=(_{a}\nabla_{\gamma}^{-\alpha})f(z)=\int_{a+1}^{z}\frac{[x_{\gamma
+\alpha-1}(z)-x_{\gamma+\alpha-1}(t-1)]^{(\alpha-1)}}{[\Gamma(\alpha)]_{q}%
}f(t)d_{\nabla}x_{\gamma}(t),
\]
then we have $g(a)=0.$

And
\[
(_{a}\nabla_{\gamma+\alpha-1})g(z)=\int_{a+1}^{z}\frac{[x_{\gamma+\alpha
-2}(z)-x_{\gamma+\alpha-2}(t-1)]^{(\alpha-2)}}{[\Gamma(\alpha-1)]_{q}%
}f(t)d_{\nabla}x_{\gamma}(t),
\]
then we have $(_{a}\nabla_{\gamma+\alpha-1})g(a)=0.$

In the same way, we have%
\[
(_{a}\nabla_{\gamma+\alpha-k}^{k})g(a)=0,k=0,1,...,m-1.
\]
Therefore, by \textbf{Proposition \ref{pro20}}, we obtain $(_{a}^{C}%
\nabla_{\gamma}^{\alpha})g(z)=(_{a}\nabla_{\gamma}^{\alpha})g(z)=f(z).$
\end{proof}

\section{Complex Variable Approach for Riemann-Liouville Fractional Difference
On Non-uniform Lattices}

In this section, we represent $k\in N^{+}$ order difference and $\alpha\in C$
order fractional difference on non-uniform lattices in terms of complex integration.

\begin{theorem}
\label{compexdef}Let $n\in N,$ $\Gamma$ be a simple closed positively oriented
contour. If $f(s)$ is analytic in simple connected domain $D$ bounded by
$\Gamma$ and $z$ is any nonzero point lies inside $D$, then%
\begin{equation}
\nabla_{\gamma-n+1}^{n}f(z)=\frac{[n]_{q}!}{2\pi i}\frac{\log q}{q^{\frac
{1}{2}}-q^{-\frac{1}{2}}}%
{\displaystyle\oint\nolimits_{\Gamma}}
\frac{f(s)\nabla x_{\gamma+1}(s)ds}{[x_{\gamma}(s)-x_{\gamma}(z)]^{(n+1)}},
\end{equation}
where $\Gamma$ enclosed the simple poles $s=z,z-1,...,z-n$ in the complex plane.
\end{theorem}

\begin{proof}
Since the set of points $\{z-i,i=0,1,...,n\}$ lie inside $D$. Hence, from the
genaralized Cauchy's integral formula, we obtain%

\begin{equation}
f(z)=\frac{1}{2\pi i}%
{\displaystyle\oint\nolimits_{\Gamma}}
\frac{f(s)x_{\gamma}^{\prime}(s)ds}{[x_{\gamma}(s)-x_{\gamma}(z)],}%
\end{equation}
and it yields%

\begin{equation}
f(z-1)=\frac{1}{2\pi i}%
{\displaystyle\oint\nolimits_{\Gamma}}
\frac{f(s)x_{\gamma}^{\prime}(s)ds}{[x_{\gamma}(s)-x_{\gamma}(z-1)].}%
\end{equation}
Substitutig with the value of $f(z)$ and $f(z-1)$ into$\frac{\nabla
f(z)}{\nabla x_{\gamma}(z)}=\frac{f(z)-f(z-1)}{x_{\gamma}(z)-x_{\gamma}%
(z-1)},$ then we have%
\begin{align*}
\frac{\nabla f(z)}{\nabla x_{\gamma}(z)}  &  =\frac{1}{2\pi i}%
{\displaystyle\oint\nolimits_{\Gamma}}
\frac{f(s)x_{\gamma}^{\prime}(s)ds}{[x_{\gamma}(s)-x_{\gamma}(z)][x_{\gamma
}(s)-x_{\gamma}(z-1)]}\\
&  =\frac{1}{2\pi i}%
{\displaystyle\oint\nolimits_{\Gamma}}
\frac{f(s)x_{\gamma}^{\prime}(s)ds}{[x_{\gamma}(s)-x_{\gamma}(z)]^{(2)}}.
\end{align*}
Substitutig with the value of $\frac{\nabla f(z)}{\nabla x_{\gamma}(z)}$ and
$\frac{\nabla f(z-1)}{\nabla x_{\gamma}(z-1)}$ into $\frac{\frac{\nabla
f(z)}{\nabla x_{\gamma}(z)}-\frac{\nabla f(z-1)}{\nabla x_{\gamma}(z-1)}%
}{x_{\gamma}(z)-x_{\gamma}(z-2)},$ then we have%
\[
\frac{\frac{\nabla f(z)}{\nabla x_{\gamma}(z)}-\frac{\nabla f(z-1)}{\nabla
x_{\gamma}(z-1)}}{x_{\gamma}(z)-x_{\gamma}(z-2)}=\frac{1}{2\pi i}%
{\displaystyle\oint\nolimits_{\Gamma}}
\frac{f(s)x_{\gamma}^{\prime}(s)ds}{[x_{\gamma}(s)-x_{\gamma}(z)]^{(3)}}.
\]
In view of
\[
x_{\gamma}(z)-x_{\gamma}(z-2)=[2]_{q}\nabla x_{\gamma-1}(z),
\]
we obtain%
\[
\frac{\nabla}{\nabla x_{\gamma-1}(z)}(\frac{\nabla f(z)}{\nabla x_{\gamma}%
(z)})=\frac{[2]_{q}}{2\pi i}%
{\displaystyle\oint\nolimits_{\Gamma}}
\frac{f(s)x_{\gamma}^{\prime}(s)ds}{[x_{\gamma}(s)-x_{\gamma}(z)]^{(3)}}.
\]
More generalaly, by the induction, we can obtain%
\[
\frac{\nabla}{\nabla x_{\gamma-n+1}(z)}(\frac{\nabla}{\nabla x_{\gamma
-n+2}(z)}...(\frac{\nabla f(z)}{\nabla x_{\gamma}(z)}))=\frac{[n]_{q}!}{2\pi
i}%
{\displaystyle\oint\nolimits_{\Gamma}}
\frac{f(s)x_{\gamma}^{\prime}(s)ds}{[x_{\gamma}(s)-x_{\gamma}(z)]^{(n+1)}},
\]
where%

\[
\lbrack x_{\gamma}(s)-x_{\gamma}(z)]^{(n+1)}=%
{\displaystyle\prod\limits_{i=0}^{n}}
[x_{\gamma}(s)-x_{\gamma}(z-i)].
\]
And last, by the use of identity%

\[
x_{\gamma}^{\prime}(s)=\frac{\log q}{q^{\frac{1}{2}}-q^{-\frac{1}{2}}}\nabla
x_{\gamma+1}(s),
\]
we have%

\begin{equation}
\nabla_{\gamma-n+1}^{n}f(z)=\frac{[n]_{q}!}{2\pi i}\frac{\log q}{q^{\frac
{1}{2}}-q^{-\frac{1}{2}}}%
{\displaystyle\oint\nolimits_{\Gamma}}
\frac{f(s)\nabla x_{\gamma+1}(s)ds\xi}{[x_{\gamma}(s)-x_{\gamma}(z)]^{(n+1)}.}
\label{diff}%
\end{equation}

\end{proof}

Inspired by formula (\ref{diff}), so we can give the definition of fractional
difference of $f(z)$ over $\{a+1, a+2,..., z\}$ on non-uniform lattices as follows

\begin{definition}
(Complex fractional difference on non-uniform lattices) \label{complexdef1}Let
$\Gamma$ be a simple closed positively oriented contour. If $f(s)$ is analytic
in simple connected domain $D$ bounded by $\Gamma$, assume that $z$ is any
nonzero point inside $D$, $a+1$ is a point inside $D$, and $z-a\in N$, then
for any $\alpha\in R^{+},$ the $\alpha$-th order fractional difference of
$f(z)$ over $\{a+1,a+2,...,z\}$ on non-uniform lattices is defined by%
\begin{equation}
\nabla_{\gamma-\alpha+1}^{\alpha}f(z)=\frac{[\Gamma(\alpha+1)]_{q}}{2\pi
i}\frac{\log q}{q^{\frac{1}{2}}-q^{-\frac{1}{2}}}%
{\displaystyle\oint\nolimits_{\Gamma}}
\frac{f(s)\nabla x_{\gamma+1}(s)ds}{[x_{\gamma}(s)-x_{\gamma}(z)]^{(\alpha
+1)}}, \label{comd}%
\end{equation}
where $\Gamma$ enclosed the simple poles $s=z,z-1,...,a+1$ in the complex plane.
\end{definition}

We can calculate the integral (\ref{comd}) by Cauchy's residue theorem. In
detail, we have

\begin{theorem}
(Fractional difference on non-uniform lattices)\label{complexdef2}Assume
$z,a\in C,z-a\in N,\alpha\in R^{+}$.

(1) Let $x(s)$ be quadratic lattices (\ref{non2}), then the $\alpha$-th order
fractional difference of $f(z)$ over $\{a+1,a+2,...,z\}$ on non-uniform
lattices can be rewritten by
\begin{equation}
\nabla_{\gamma+1-\alpha}^{\alpha}[f(z)]=\sum_{k=0}^{z-(a+1)}f(z-k)\frac
{\Gamma(2z-k+\gamma-\alpha)\nabla x_{\gamma+1}(z-k)}{\Gamma(2z+\gamma
+1-k)}\frac{(-\alpha)_{k}}{k!}; \label{sumd1}%
\end{equation}

(2) Let $x(s)$ be quadratic lattices(\ref{non1}), then the $\alpha$-th order
fractional difference of $f(z)$ over $\{a+1,a+2,...,z\}$ on non-uniform
lattices can be rewritten by
\begin{equation}
\nabla_{\gamma+1-\alpha}^{\alpha}[f(z)]=\sum_{k=0}^{z-(a+1)}f(z-k)\frac
{[\Gamma(2z-k+\gamma-\alpha)]_{q}\nabla x_{\gamma+1}(z-k)}{[\Gamma
(2z+\gamma+1-k)]_{q}}\frac{([-\alpha]_{q})_{k}}{[k]_{q}!}. \label{sumd2}%
\end{equation}

\end{theorem}

\begin{proof}
From (\ref{comd}), in quadratic lattices (\ref{non2}), one has%
\begin{align*}
\nabla_{\gamma+1-\alpha}^{\alpha}[f(z)]  &  =\frac{\Gamma(\alpha+1)}{2\pi i}%
{\displaystyle\oint\nolimits_{\Gamma}}
\frac{f(s)\nabla x_{\gamma+1}(s)ds}{[x_{\gamma}(s)-x_{\gamma}(z)^{(\alpha+1)}%
}\\
&  =\frac{\Gamma(\alpha+1)}{2\pi i}%
{\displaystyle\oint\nolimits_{\Gamma}}
\frac{f(s)\nabla x_{\gamma+1}(s)\Gamma(s-z)\Gamma(s+z+\gamma-\alpha)ds}%
{\Gamma(s-z+\alpha+1)\Gamma(s+z+\gamma+1)}.
\end{align*}
According to the assumption of \textbf{Definition} \ref{complexdef1},
$\Gamma(s-z)$ has simple poles at $s=z-k,k=0,1,2,...,z-(a+1)$. The residue of
$\Gamma(s-z)$ at the point $s-z=-k$ is%
\begin{align*}
&  \lim_{s\rightarrow z-k}(s-z+k)\Gamma(s-z)\\
&  =\lim_{s\rightarrow z-k}\frac{(s-z)(s-z+1)...(s-z+k-1)(s-z+k)\Gamma
(s-z)}{(s-z)(s-z+1)...(s-z+k-1)}\\
&  =\lim_{s\rightarrow z-k}\frac{\Gamma(s-z+k+1)}{(s-z)(s-z+1)...(s-z+k-1)}\\
&  =\frac{1}{(-k)(-k+1)...(-1)}=\frac{(-1)^{k}}{k!}.
\end{align*}
Then by the use of Cauchy's residue theorem, we have%
\[
\nabla_{\gamma+1-\alpha}^{\alpha}[f(z)]=\Gamma(\alpha+1)\sum_{k=0}%
^{z-(a+1)}f(z-k)\frac{\Gamma(2z-k+\gamma-\alpha)\nabla x_{\gamma+1}%
(z-k)}{\Gamma(\alpha+1-k)\Gamma(2z+\gamma+1-k)}\frac{(-1)^{k}}{k!}.
\]
Since%
\[
\frac{\Gamma(\alpha+1)}{\Gamma(\alpha+1-k)}=\alpha(\alpha-1)...(\alpha-k+1),
\]
and%
\[
\alpha(\alpha-1)...(\alpha-k+1)(-1)^{k}=(-\alpha)_{k},
\]
therefore, we get%
\[
\nabla_{\gamma+1-\alpha}^{\alpha}[f(z)]=\sum_{k=0}^{z-(a+1)}f(z-k)\frac
{\Gamma(2z-k+\gamma-\alpha)\nabla x_{\gamma+1}(z-k)}{\Gamma(2z+\gamma
+1-k)}\frac{(-\alpha)_{k}}{k!}.
\]
\bigskip From (\ref{comd}), in quadratic lattices (\ref{non1}), we have%
\begin{align}
\nabla_{\gamma-\alpha+1}^{\alpha}f(z)  &  =\frac{[\Gamma(\alpha+1)]_{q}}{2\pi
i}\frac{\log q}{q^{\frac{1}{2}}-q^{-\frac{1}{2}}}%
{\displaystyle\oint\nolimits_{\Gamma}}
\frac{f(s)\nabla x_{\gamma+1}(s)ds}{[x_{\gamma}(s)-x_{\gamma}(z)]^{(\alpha
+1)}}\nonumber\\
&  =\frac{[\Gamma(\alpha+1)]_{q}}{2\pi i}\frac{\log q}{q^{\frac{1}{2}%
}-q^{-\frac{1}{2}}}%
{\displaystyle\oint\nolimits_{\Gamma}}
\frac{f(s)\nabla x_{\gamma+1}(s)[\Gamma(s-z)]_{q}[\Gamma(s+z+\gamma
-\alpha)]_{q}ds}{[\Gamma(s-z+\alpha+1)]_{q}[\Gamma(s+z+\gamma+1)]_{q}}%
\end{align}
From the assumption of \textbf{Definition} \ref{complexdef1}, $[\Gamma
(s-z)]_{q}$ has simple poles at $s=z-k,k=0,1,2,...,z-(a+1)$. The residue of
$[\Gamma(s-z)]_{q}$ at the point $s-z=-k$ is%
\begin{align*}
&  \lim_{s\rightarrow z-k}(s-z+k)[\Gamma(s-z)]_{q}\\
&  =\lim_{s\rightarrow z-k}\frac{s-z+k}{[s-z+k]_{q}}[s-z+k]_{q}[\Gamma
(s-z)]_{q}\\
&  =\frac{q^{\frac{1}{2}}-q^{-\frac{1}{2}}}{\log q}\lim_{s\rightarrow
z-k}[s-z+k]_{q}[\Gamma(s-z)]_{q}\\
&  =\frac{q^{\frac{1}{2}}-q^{-\frac{1}{2}}}{\log q}\lim_{s\rightarrow
z-k}\frac{[s-z]_{q}[s-z+1]_{q}...[s-z+k-1]_{q}[s-z+k]_{q}[\Gamma(s-z)]_{q}%
}{(s-z)(s-z+1)...(s-z+k-1)}\\
&  =\frac{q^{\frac{1}{2}}-q^{-\frac{1}{2}}}{\log q}\lim_{s\rightarrow
z-k}\frac{[\Gamma(s-z+k+1)]_{q}}{[s-z]_{q}[s-z+1]_{q}...[s-z+k-1]_{q}}\\
&  =\frac{q^{\frac{1}{2}}-q^{-\frac{1}{2}}}{\log q}\frac{1}{[-k]_{q}%
[-k+1]_{q}...[-1]_{q}}=\frac{q^{\frac{1}{2}}-q^{-\frac{1}{2}}}{\log q}%
\frac{(-1)^{k}}{[k]_{q}!}.
\end{align*}
Then by the use of Cauchy's residue theorem, we have%

\[
\nabla_{\gamma+1-\alpha}^{\alpha}[f(z)]=[\Gamma(\alpha+1)]_{q}\sum
_{k=0}^{z-(a+1)}f(z-k)\frac{[\Gamma(2z-k+\gamma-\alpha)]_{q}\nabla
x_{\gamma+1}(z-k)}{[\Gamma(\alpha+1-k)]_{q}[\Gamma(2z+\gamma+1-k)]_{q}}%
\frac{(-1)^{k}}{[k]_{q}!}.
\]
Since%

\[
\frac{\lbrack\Gamma(\alpha+1)]_{q}}{[\Gamma(\alpha+1-k)]_{q}}=[\alpha
]_{q}[\alpha-1]_{q}...[\alpha-k+1]_{q},
\]
and%
\[
\lbrack\alpha]_{q}[\alpha-1]_{q}...[\alpha-k+1](-1)^{k}=([-\alpha])_{k},
\]
therefore, we obtain that%
\[
\nabla_{\gamma+1-\alpha}^{\alpha}[f(z)]=\sum_{k=0}^{z-(a+1)}f(z-k)\frac
{[\Gamma(2z-k+\gamma-\alpha)]_{q}\nabla x_{\gamma+1}(z-k)}{[\Gamma
(2z+\gamma+1-k)]_{q}}\frac{([-\alpha]_{q})_{k}}{k!}.
\]

\end{proof}

So far, with respect to the definition of the R-L fractional difference on
non-uniform lattices, we have given two kinds of definitions, such as
\textbf{Definition} \ref{maindef2} or \textbf{Definition} \ref{maindef3} in
section 4 and \textbf{Definition} \ref{complexdef1} or \textbf{Definition}
\ref{complexdef2} in section 7 through two different ideas and methods. Now
let's compare \textbf{Definition} \ref{maindef3} in section 4 and
\textbf{Definition} \ref{complexdef2} in section 7.

Here follows a theorem connecting the R-L fractional difference (\ref{adiff4})
and the complex generalization of fractional difference (\ref{comd}) :

\begin{theorem}
\label{equa}For any $\alpha\in R^{+},$ let $\Gamma$ be a simple closed
positively oriented contour. If $f(s)$ is analytic in simple connected domain
$D$ bounded by $\Gamma$, assume that $z$ is any nonzero point inside $D$,
$a+1$ is a point inside $D$, such that $z-a\in N$, then the complex
generalization fractional integral (\ref{comd}) equals the R-L fractional
defference (\ref{adiff3}) or (\ref{adiff4}):%
\[
\nabla_{\gamma+1-\alpha}^{\alpha}[f(z)]=\sum_{k=a+1}^{z}\frac{[x_{\gamma
-\alpha}(z)-x_{\gamma-\alpha}(k-1)]^{(-\alpha-1)}}{[\Gamma(-\alpha)]_{q}%
}f(k)\nabla x_{\gamma+1}(k).
\]

\end{theorem}

\begin{proof}
By \textbf{Theorem} \ref{complexdef2}, we have%
\begin{align*}
\nabla_{\gamma+1-\alpha}^{\alpha}[f(z)]  &  =\sum_{k=0}^{z-(a+1)}%
\frac{([-\alpha]_{q})_{k}}{[k]_{q}!}\frac{[\Gamma(2z-k+\gamma-\alpha)]_{q}%
}{[\Gamma(2z-k+\gamma+1)]_{q}}f(z-k)\nabla x_{\gamma+1}(z-k).\\
&  =\sum_{k=0}^{z-(a+1)}\frac{[\Gamma(k-\alpha)]_{q}}{[\Gamma(-\alpha
)]_{q}[\Gamma(k+1)]_{q}}\frac{[\Gamma(2z-k+\gamma-\alpha)]_{q}}{[\Gamma
(2z-k+\gamma+1)]_{q}}f(z-k)\nabla x_{\gamma+1}(z-k)\\
&  =\sum_{k=0}^{z-(a+1)}\frac{[x_{\gamma-\alpha}(z)-x_{\gamma-\alpha
}(z-k-1)]^{(-\alpha-1)}}{[\Gamma(-\alpha)]_{q}}f(z-k)\nabla x_{\gamma
+1}(z-k)\\
&  =\sum_{k=a+1}^{z}\frac{[x_{\gamma-\alpha}(z)-x_{\gamma-\alpha
}(k-1)]^{(-\alpha-1)}}{[\Gamma(-\alpha)]_{q}}f(k)\nabla x_{\gamma+1}(k).
\end{align*}
So that the two \textbf{Definition} \ref{maindef3} and \textbf{Definition}
\ref{complexdef2} are consistent.
\end{proof}

\ Set $\alpha=\gamma$ in \textbf{Theorem} \ref{complexdef2}, we obtain

\begin{corollary}
\label{cor34}Assume that conditions of \textbf{Definition} \ref{complexdef1}
hold, then%
\begin{align*}
\nabla_{1}^{\gamma}[f(z)]  &  =\frac{[\Gamma(\gamma+1)]_{q}}{2\pi i}\frac{\log
q}{q^{\frac{1}{2}}-q^{-\frac{1}{2}}}%
{\displaystyle\oint\nolimits_{\Gamma}}
\frac{f(s)\nabla x_{\gamma+1}(s)ds}{[x_{\gamma}(s)-x_{\gamma}(z)^{(\gamma+1)}%
}\\
&  =\sum_{k=0}^{z-(a+1)}f(z-k)\frac{[\Gamma(2z+\mu-k)]_{q}\nabla x_{\gamma
+1}(z-k)}{[\Gamma(2z+\gamma+\mu+1-k)]_{q}}\frac{([-\gamma]_{q})_{k}}{[k]_{q}%
!}.
\end{align*}
where $\Gamma$ enclosed the simple poles $s=z,z-1,...,a+1$ in the complex plane.
\end{corollary}

\begin{remark}
\label{rem35}When $\gamma=n\in N^{+},$we have%
\begin{align}
\nabla_{1}^{n}[f(z)]  &  =\frac{[\Gamma(n+1)]_{q}}{2\pi i}\frac{\log
q}{q^{\frac{1}{2}}-q^{-\frac{1}{2}}}%
{\displaystyle\oint\nolimits_{\Gamma}}
\frac{f(s)\nabla x_{\gamma+1}(s)ds}{[x_{n}(s)-x_{n}(z)^{(n+1)}}\nonumber\\
&  =\sum_{k=0}^{n}f(z-k)\frac{[\Gamma(2z+\mu-k)]_{q}\nabla x_{n+1}%
(z-k)}{[\Gamma(2z+n+\mu+1-k)]_{q}}\frac{([-n]_{q})_{k}}{k!},
\end{align}
where $\Gamma$ enclosed the simple poles $s=z,z-1,...,z-n$ in the complex plane.

This is consistent with \textbf{Definition} \ref{def8} proposed by Nikiforov.
A, Uvarov. V, Suslov. S in \cite{nikiforov1991} .
\end{remark}

\ For complex integral of Riemann-Liouville fractional difference on
uon-uniform lattices, we can establish a important Cauchy Beta formula as follows:

\begin{theorem}
(Cauchy Beta formula)\label{cauchybeta} Let $\alpha,\beta\in C,$ and assume
that
\[%
{\displaystyle\oint\nolimits_{\Gamma}}
\Delta_{t}\{\frac{1}{[x_{\beta}(z)-x_{\beta}(t)]^{(\beta)}}\frac{1}%
{[x_{-1}(t)-x_{-1}(a)]^{(\alpha-1)}}\}dt=0,
\]
then%
\[
\frac{1}{2\pi i}\frac{\log q}{q^{\frac{1}{2}}-q^{-\frac{1}{2}}}%
{\displaystyle\oint\nolimits_{\Gamma}}
\frac{[\Gamma(\beta+1)]_{q}}{[x_{\beta}(z)-x_{\beta}(t)]^{(\beta+1)}}%
\frac{[\Gamma(\alpha)]_{q}\Delta y_{-1}(t)dt}{[x(t)-x(a)]^{(\alpha)}}%
=\frac{[\Gamma(\alpha+\beta)]_{q}}{[x_{\beta}(z)-x_{\beta}(a)]^{(\alpha
+\beta)}},
\]
where $\Gamma$ be a simple closed positively oriented contour, $a$ lies inside
$C.$
\end{theorem}

In order to prove \textbf{Theorem} \ref{cauchybeta}, we first give a lemma.

\begin{lemma}
\label{lem37}For any $\alpha,\beta,$ then we have%
\begin{align}
&  [1-\alpha]_{q}[x_{\beta}(z)-x_{\beta}(t-\beta)]+[\beta]_{q}[x_{1-\alpha
}(t+\alpha-1)-x_{1-\alpha}(a)]\nonumber\\
&  =[1-\alpha]_{q}[x_{\beta}(z)-x_{\beta}(a+1-\alpha-\beta)]+[\alpha
+\beta-1]_{q}[x(t)-x(a+1-\alpha)]. \label{ieq1}%
\end{align}

\end{lemma}

\begin{proof}
(\ref{ieq1}) is equivalent to%
\begin{align}
&  [\alpha+\beta-1]_{q}x(t)+[1-\alpha]_{q}x_{\beta}(t-\beta)-[\beta
]_{q}x_{1-\alpha}(t+\alpha-1)\nonumber\\
&  =[\alpha+\beta-1]_{q}x(a+1-\alpha)+[1-\alpha]_{q}x_{\beta}(a+1-\alpha
-\beta)-[\beta]_{q}x_{1-\alpha}(a). \label{ieq2}%
\end{align}

Set $\alpha-1=\widetilde{\alpha}$, then (\ref{ieq2}) can be written as%
\begin{align}
&  [\widetilde{\alpha}+\beta]_{q}x(t)-[\widetilde{\alpha}]_{q}x_{-\beta
}(t)-[\beta]_{q}x_{\widetilde{\alpha}}(t)\nonumber\\
&  =[\widetilde{\alpha}+\beta]_{q}x(a-\widetilde{\alpha})-[\widetilde{\alpha
}]_{q}x_{-\beta}(a-\widetilde{\alpha})-[\beta]_{q}x_{\widetilde{\alpha}%
}(a-\widetilde{\alpha}). \label{ieq3}%
\end{align}
By the use of \textbf{Lemma} \ref{lem13}, then Eq. (\ref{ieq3}) holds, and
then Eq. (\ref{ieq1}) holds.
\end{proof}

\textbf{Proof of Theorem} \ref{cauchybeta}: Set%

\[
\rho(t)=\frac{1}{[x_{\beta}(z)-x_{\beta}(t)]^{(\beta+1)}}\frac{1}%
{[x(t)-x(a)]^{(\alpha)}},
\]
and%
\[
\sigma(t)=[x_{\alpha-1}(t+\alpha-1)-x_{\alpha-1}(a)][x_{\beta}(z)-x_{\beta
}(t)].
\]
Since%

\[
\lbrack x_{\beta}(z)-x_{\beta}(t)]^{(\beta+1)}=[x_{\beta}(z)-x_{\beta
}(t-1)]^{(\beta)}[x_{\beta}(z)-x_{\beta}(t)],
\]
and%

\[
\lbrack x(t)-x(a)]^{(\alpha)}=[x_{-1}(t)-x_{-1}(a)]^{(\alpha-1)}[x_{\alpha
-1}(t+\alpha-1)-x_{\alpha-1}(a)],
\]
these reduce to%

\[
\sigma(t)\rho(t)=\frac{1}{[x_{\beta}(z)-x_{\beta}(t)]^{(\beta)}}\frac
{1}{[x_{-1}(t)-x_{-1}(a)]^{(\alpha-1)}}.
\]
Making use of%

\[
\Delta_{t}[f(t)g(t)]=g(t+1)\Delta_{t}[f(t)]+f(t)\Delta_{t}[g(t)],
\]
where%

\[
f(t)=\frac{1}{[x_{-1}(t)-x_{-1}(a)]^{(\alpha-1)}},g(t)=\frac{1}{[x_{\beta
}(z)-x_{\beta}(t-1)]^{(\beta)}},
\]
and%

\[
\frac{\Delta_{t}}{\Delta x_{-1}(t)}\{\frac{1}{[x_{-1}(t)-x_{-1}(a)]^{(\alpha
-1)}}\}=\frac{[1-\alpha]_{q}}{[x(t)-x(a)]^{(\alpha)}},
\]

\begin{align*}
&  \frac{\Delta_{t}}{\Delta x_{-1}(t)}\{\frac{1}{[x_{\beta}(z)-x_{\beta
}(t-1)]^{(\beta)}}\}\\
&  =\frac{\nabla_{t}}{\nabla x_{1}(t)}\{\frac{1}{[x_{\beta}(z)-x_{\beta
}(t)]^{(\beta)}}\}\\
&  =\frac{[\beta]_{q}}{[x_{\beta}(z)-x_{\beta}(t)]^{(\beta+1)}}.
\end{align*}
then, we have%
\begin{align*}
&  \frac{\Delta_{t}}{\Delta x_{-1}(t)}\{\sigma(t)\rho(t)\}\\
&  =\frac{1}{[x_{\beta}(z)-x_{\beta}(t)]^{(\beta)}}\frac{[1-\alpha]_{q}%
}{[x(t)-x(a)]^{(\alpha)}}+\\
&  +\frac{1}{[x_{-1}(t)-x_{-1}(a)]^{(\alpha-1)}}\frac{[\beta]_{q}}{[x_{\beta
}(z)-x_{\beta}(t)]^{(\beta+1)}}\\
&  =\{[1-\alpha]_{q}[x_{\beta}(z)-x_{\beta}(t-\beta)]+[\beta]_{q}[x_{1-\alpha
}(t+\alpha-1)-x_{1-\alpha}(a)]\}\\
&  \times\frac{1}{[x(t)-x(a)]^{(\alpha)}}\frac{1}{[x_{\beta}(z)-x_{\beta
}(t)]^{(\beta+1)}}\\
&  =\tau(t)\rho(t),
\end{align*}
where%

\[
\tau(t)=[1-\alpha]_{q}[x_{\beta}(z)-x_{\beta}(t-\beta)]+[\beta]_{q}%
[x_{1-\alpha}(t+\alpha-1)-x_{1-\alpha}(a)],
\]
this is due to%

\[
\lbrack x_{\beta}(z)-x_{\beta}(t)]^{(\beta+1)}=[x_{\beta}(z)-x_{\beta
}(t)]^{(\beta)}[x_{\beta}(z)-x_{\beta}(t-\beta)].
\]
From \textbf{Proposition} \ref{properties} one has%

\begin{align*}
&  \frac{\Delta_{t}}{\Delta x_{-1}(t)}\{\sigma(t)\rho(t)\}\\
&  =\{[1-\alpha]_{q}[x_{\beta}(z)-x_{\beta}(a+1-\alpha-\beta)]+[\alpha
+\beta-1]_{q}[x(t)-x(a+1-\alpha)]\}\\
&  \cdot\frac{1}{[x_{\beta}(z)-x_{\beta}(t)]^{(\beta+1)}}\frac{1}%
{[x(t)-x(a)]^{(\alpha)}},
\end{align*}
Or%
\begin{align}
&  \Delta_{t}\{\sigma(t)\rho(t)\}\nonumber\\
&  =\{[1-\alpha]_{q}[x_{\beta}(z)-x_{\beta}(a+1-\alpha-\beta)]+[\alpha
+\beta-1]_{q}[x(t)-x(a+1-\alpha)]\}\nonumber\\
&  \cdot\frac{1}{[x_{\beta}(z)-x_{\beta}(t)]^{(\beta+1)}}\frac{1}%
{[x(t)-x(a)]^{(\alpha)}}\Delta x_{-1}(t). \label{ieq4}%
\end{align}
Set%

\begin{equation}
I(\alpha)=\frac{1}{2\pi i}\frac{\log q}{q^{\frac{1}{2}}-q^{-\frac{1}{2}}}%
{\displaystyle\oint\nolimits_{\Gamma}}
\frac{1}{[x_{\beta}(z)-x_{\beta}(t)]^{(\beta+1)}}\frac{\nabla y_{1}\left(
t\right)  dt}{[x(t)-x(a)]^{(\alpha)}}, \label{ie}%
\end{equation}
and%
\[
I(\alpha-1)=\frac{1}{2\pi i}\frac{\log q}{q^{\frac{1}{2}}-q^{-\frac{1}{2}}}%
{\displaystyle\oint\nolimits_{\Gamma}}
\frac{1}{[x_{\beta}(z)-x_{\beta}(t)]^{(\beta+1)}}\frac{\nabla y_{1}\left(
t\right)  dt}{[x(t)-x(a)]^{(\alpha-1)}}.
\]
Since%

\[
\lbrack x(t)-x(a)]^{(\alpha-1)}[x(t)-x(a+1-\alpha)]=[x(t)-x(a)]^{(\alpha)},
\]
then%

\[
I(\alpha-1)=\frac{1}{2\pi i}\frac{\log q}{q^{\frac{1}{2}}-q^{-\frac{1}{2}}}%
{\displaystyle\oint\nolimits_{\Gamma}}
\frac{1}{[x_{\beta}(z)-x_{\beta}(t)]^{(\beta+1)}}\frac{[x(t)-x(a+1-\alpha
)]\nabla x_{1}\left(  t\right)  dt}{[x(t)-x(a)]^{(\alpha)}}.
\]
Integrating both sides of equation (\ref{ieq4}), then we have%
\begin{align*}%
{\displaystyle\oint\nolimits_{\Gamma}}
\Delta_{t}\{\sigma(t)\rho(t)\}dt  &  =[1-\alpha]_{q}[x_{\beta}(z)-x_{\beta
}(a+1-\alpha-\beta)]I(\alpha)\\
&  -[\alpha+\beta-1]_{q}I(\alpha-1).
\end{align*}
If
\[%
{\displaystyle\oint\nolimits_{\Gamma}}
\Delta_{t}\{\sigma(t)\rho(t)\}dt=0,
\]
then we obtain that%
\[
\frac{I(\alpha-1)}{I(\alpha)}=\frac{[\alpha-1]_{q}}{[\alpha+\beta-1]_{q}%
}[y_{\beta}(z)-y_{\beta}(a+1-\alpha-\beta)].
\]
That is%

\begin{equation}
\frac{I(\alpha-1)}{I(\alpha)}=\frac{\frac{[\Gamma(\alpha+\beta-1)]_{q}%
}{[\Gamma(\alpha-1)]_{q}}}{\frac{[\Gamma(\alpha+\beta)]_{q}}{[\Gamma
(\alpha)]_{q}}}\frac{\frac{1}{[x_{\beta}(z)-x_{\beta}(a)]^{(\alpha+\beta-1)}}%
}{\frac{1}{[x_{\beta}(z)-x_{\beta}(a)]^{(\alpha+\beta)}}}. \label{ieq5}%
\end{equation}
From (\ref{ieq5}), we set%
\begin{equation}
I(\alpha)=k\frac{[\Gamma(\alpha+\beta)]_{q}}{[\Gamma(\alpha)]_{q}}\frac
{1}{[x_{\beta}(z)-x_{\beta}(a)]^{(\alpha+\beta)}}, \label{ie1}%
\end{equation}
where $k$ is undetermined.

Set $\alpha=1$, one has%

\begin{equation}
I(1)=k[\Gamma(1+\beta)]_{q}\frac{1}{[x_{\beta}(z)-x_{\beta}(a)]^{(1+\beta)}},
\label{i3}%
\end{equation}
and from (\ref{ie}) and generalized Cauchy residue theorem, one has%

\begin{align}
I(1)  &  =\frac{1}{2\pi i}\frac{\log q}{q^{\frac{1}{2}}-q^{-\frac{1}{2}}}%
{\displaystyle\oint\nolimits_{\Gamma}}
\frac{1}{[x_{\beta}(z)-x_{\beta}(t)]^{(\beta+1)}}\frac{\nabla x_{1}\left(
t\right)  dt}{[x(t)-x(a)]^{(1)}}\nonumber\\
&  =\frac{1}{2\pi i}%
{\displaystyle\oint\nolimits_{\Gamma}}
\frac{1}{[x_{\beta}(z)-x_{\beta}(t)]^{(\beta+1)}}\frac{x^{\prime}\left(
t\right)  dt}{[x(t)-x(a)]}\nonumber\\
&  =\frac{1}{[x_{\beta}(z)-x_{\beta}(a)]^{(\beta+1)}}, \label{i4}%
\end{align}
From (\ref{i3}) and (\ref{i4}), we get%

\[
k=\frac{1}{[\Gamma(1+\beta)]_{q}}.
\]
Therefore, we obtain that%

\[
I(\alpha)=\frac{[\Gamma(\alpha+\beta)]_{q}}{[\Gamma(\beta+1)]_{q}%
[\Gamma(\alpha)]_{q}}\frac{1}{[x_{\beta}(z)-x_{\beta}(a)]^{(\alpha+\beta)}},
\]
and \textbf{Theorem} \ref{cauchybeta} is completed.

\section{Fractional Central Sum and Difference on Non-uniform Lattices}

Next we will give the difinition of fractional central sum and fractional
central difference on Non-uniform Lattices. Let us first give the integral
central sum of $f(z)$ on non-uniform lattices $x(s).$

$1-th$ central sum of $f(z)$ on non-uniform lattices $x(s)$ is defined by%
\[
\delta_{0}^{-1}f(z)=y_{1}(z)=\sum_{s=a+\frac{1}{2}}^{z-\frac{1}{2}}f(s)\delta
x(s)=\int_{a+\frac{1}{2}}^{z-\frac{1}{2}}f(s)d_{\delta}x(s),
\]
where $f(s)$ is defined in $\{a+\frac{1}{2},\operatorname{mod}(1)\},$and
$y_{1}(z)$ is defined in $\{a+1,\operatorname{mod}(1)\}.$

Then we have $2-th$ central sum of $f(z)$ on non-uniform lattices $x(s)$ as follows%

\begin{align*}
\delta_{0}^{-2}f(z)  &  =y_{2}(z)=\int_{a+1}^{z-\frac{1}{2}}y_{1}(s)d_{\delta
}x(s)\\
&  =\int_{a+1}^{z-\frac{1}{2}}d_{\delta}x(s)\int_{a+\frac{1}{2}}^{s-\frac
{1}{2}}f(t)d_{\delta}x(t)\\
&  =\int_{a+\frac{1}{2}}^{z-1}f(t)d_{\delta}x(t)\int_{t+\frac{1}{2}}%
^{z-\frac{1}{2}}d_{\delta}x(s)\\
&  =\int_{a+\frac{1}{2}}^{z-1}\frac{[x(z)-x(t)]}{[\Gamma(2)]_{q}}%
f(t)d_{\delta}x(t).
\end{align*}
where $y_{1}(s)$ is defined in $\{a+1,\operatorname{mod}(1)\},$and $y_{2}(z)$
is defined in $\{a+\frac{3}{2},\operatorname{mod}(1)\}.$ and $3-th$ central
sum of $f(z)$ on non-uniform lattices $x(s)$ is%

\begin{align*}
\delta_{0}^{-3}f(z)  &  =y_{3}(z)=\int_{a+\frac{3}{2}}^{z-\frac{1}{2}}%
y_{2}(s)d_{\delta}x(s)\\
&  =\int_{a+\frac{3}{2}}^{z-\frac{1}{2}}d_{\delta}x(s)\int_{a+\frac{1}{2}%
}^{s-1}[x(s)-(t)]f(t)d_{\delta}x(t)\\
&  =\int_{a+\frac{1}{2}}^{z-\frac{3}{2}}f(t)d_{\delta}x(t)\int_{t+1}%
^{z-\frac{1}{2}}[x(s)-x(t)]d_{\delta}x(s)\\
&  =\int_{a+\frac{1}{2}}^{z-\frac{3}{2}}\frac{[x(z)-x(t)]^{(2)}}%
{[\Gamma(3)]_{q}}f(t)d_{\delta}x(t).
\end{align*}
where $y_{2}(s)$ is defined in $\{a+\frac{3}{2},\operatorname{mod}(1)\},$and
$y_{3}(z)$ is defined in $\{a+2,\operatorname{mod}(1)\}.$

More generalaly, we have $k-th$ central sum of $f(z)$ on non-uniform lattices
$x(s)$ as follows%

\begin{align}
\delta^{-k}f(z)  &  =y_{k}(z)=\int_{a+\frac{k}{2}}^{z-\frac{1}{2}}%
y_{k-1}(s)d_{\delta}x(s)\nonumber\\
&  =\int_{a+\frac{1}{2}}^{z-\frac{k}{2}}\frac{[x(z)-x_{k-2}(t)]^{(k-1)}%
}{[\Gamma(k)]_{q}}f(t)d_{\delta}x(t).
\end{align}
where $y_{k-1}(s)$ is defined in $\{a+\frac{k}{2},\operatorname{mod}(1)\},$and
$y_{k}(z)$ is defined in $\{a+\frac{k+1}{2},\operatorname{mod}(1)\}.$

\begin{definition}
\label{sumcentral}For any $\operatorname{Re}\alpha\in R^{+},$ the $\alpha-th$
fractional central sum of $f(z)$ on non-uniform lattices $x(s)$ is defined by
\end{definition}

\begin{equation}
\delta_{0}^{-\alpha}f(z)=\int_{a+\frac{1}{2}}^{z-\frac{\alpha}{2}}%
\frac{[x(z)-x_{\alpha-2}(t)]^{(\alpha-1)}}{[\Gamma(\alpha)]_{q}}f(t)d_{\delta
}x(t).
\end{equation}
where $\delta^{-\alpha}f(z)$ is defined in $\{a+\frac{\alpha+1}{2}%
,\operatorname{mod}(1)\},$ and $f(t)$ is defined in $\{a+\frac{1}%
{2},\operatorname{mod}(1)\}.$

\begin{definition}
\label{def40}Let $\delta f(z)=f(z+\frac{1}{2})-f(z-\frac{1}{2}),\delta
x(z)=x(z+\frac{1}{2})-x(z-\frac{1}{2}),$ the central difference of $f(z)$ on
$x(z)$ is defined by%
\begin{equation}
\delta_{0}f(z)=\frac{\delta f(z)}{\delta x(z)}=\frac{f(z+\frac{1}%
{2})-f(z-\frac{1}{2})}{x(z+\frac{1}{2})-x(z-\frac{1}{2})},
\end{equation}
and
\begin{equation}
\delta_{0}^{m}f(z)=\delta_{0}[\delta_{0}^{m-1}f(z)],m=1,2,...
\end{equation}

\end{definition}

\begin{definition}
\label{rlcentral}Let $m$ be the smallest integer exceeding $\alpha$, \ the
Riemann-Liouville central fractional difference is defined by%
\begin{equation}
\delta_{0}^{\alpha}f(z)=\delta_{0}^{m}(\delta_{0}^{\alpha-m}f(z)). \label{rlc}%
\end{equation}
where $f(z)$ is defined in $\{a+\frac{1}{2},\operatorname{mod}(1)\},$
$(\delta_{0}^{\alpha-m}f(z))$ is defined in $\{a+\frac{m-\alpha+1}%
{2},\operatorname{mod}(1)\}$, and $\delta_{0}^{\alpha}f(z)$ is defined in
$\{a+\frac{-\alpha+1}{2},\operatorname{mod}(1)\}$
\end{definition}

Let us calculate the right hand of Eq. (\ref{rlc}). First, by
\textbf{Definition} \ref{sumcentral} and \textbf{Definition} \ref{def40}, one has%

\begin{align*}
\delta_{0}(\delta_{0}^{\alpha-m}f(z))  &  =\frac{\delta}{\delta x(z)}%
\int_{a+\frac{1}{2}}^{z-\frac{m-\alpha}{2}}\frac{[x(z)-x_{m-\alpha
-2}(t)]^{(m-\alpha-1)}}{[\Gamma(m-\alpha)]_{q}}f(t)d_{\delta}x(t)\\
&  =\frac{1}{\delta x(z)}\{\int_{a+\frac{1}{2}}^{z+\frac{1}{2}-\frac{m-\alpha
}{2}}\frac{[x(z+\frac{1}{2})-x_{m-\alpha-2}(t)]^{(m-\alpha-1)}}{[\Gamma
(m-\alpha)]_{q}}f(t)d_{\delta}x(t)\\
&  -\int_{a+\frac{1}{2}}^{z-\frac{1}{2}-\frac{m-\alpha}{2}}\frac{[x(z-\frac
{1}{2})-x_{m-\alpha-2}(t)]^{(m-\alpha-1)}}{[\Gamma(m-\alpha)]_{q}%
}f(t)d_{\delta}x(t)\}\\
&  =\frac{1}{\delta x(z)}\{\int_{a+\frac{1}{2}}^{z+\frac{1}{2}-\frac{m-\alpha
}{2}}\frac{\delta\lbrack x(z)-x_{m-\alpha-2}(t)]^{(m-\alpha-1)}}%
{[\Gamma(m-\alpha)]_{q}}f(t)d_{\delta}x(t)\\
&  +\frac{[x(z-\frac{1}{2})-x_{m-\alpha-2}(t)]^{(m-\alpha-1)}}{[\Gamma
(m-\alpha)]_{q}}f(t)\delta x(t)|_{t=z+\frac{1}{2}-\frac{m-\alpha}{2}}\}\\
&  =\frac{1}{\delta x(z)}\int_{a+\frac{1}{2}}^{z+\frac{1}{2}-\frac{m-\alpha
}{2}}\frac{\delta\lbrack x(z)-x_{m-\alpha-2}(t)]^{(m-\alpha-1)}}%
{[\Gamma(m-\alpha)]_{q}}f(t)d_{\delta}x(t)\\
&  =\int_{a+\frac{1}{2}}^{z-\frac{m-(\alpha+1)}{2}}\frac{[x(z)-x_{m-(\alpha
+1)-2}(t)]^{[m-(\alpha+1)-1]}}{[\Gamma\lbrack m-(\alpha+1)]]_{q}}%
f(t)d_{\delta}x(t).
\end{align*}
Then, one gets further%

\begin{align*}
\delta_{0}^{2}(\delta_{0}^{\alpha-m}f(z))  &  =\delta_{0}[\delta_{0}%
(\delta_{0}^{\alpha-m}f(z))]\\
&  =\frac{\delta}{\delta x(z)}\int_{a+\frac{1}{2}}^{z-\frac{m-(\alpha+1)}{2}%
}\frac{[x(z)-x_{m-(\alpha+1)-2}(t)]^{[m-(\alpha+1)-1]}}{[\Gamma\lbrack
m-(\alpha+1)]]_{q}}f(t)d_{\delta}x(t),
\end{align*}
In the same way, we obtain%

\begin{align*}
&  \frac{\delta}{\delta x(z)}\int_{a+\frac{1}{2}}^{z-\frac{m-(\alpha+1)}{2}%
}\frac{[x(z)-x_{m-(\alpha+1)-2}(t)]^{[m-(\alpha+1)-1]}}{[\Gamma\lbrack
m-(\alpha+1)]]_{q}}f(t)d_{\delta}x(t)\\
&  =\int_{a+\frac{1}{2}}^{z-\frac{m-(\alpha+2)}{2}}\frac{[x(z)-x_{m-(\alpha
+2)-2}(t)]^{[m-(\alpha+2)-1]}}{[\Gamma\lbrack m-(\alpha+2)]]_{q}}%
f(t)d_{\delta}x(t).
\end{align*}
That is%

\[
\delta_{0}^{2}(\delta_{0}^{\alpha-m}f(z))=\int_{a+\frac{1}{2}}^{z-\frac
{m-(\alpha+2)}{2}}\frac{[x(z)-x_{m-(\alpha+2)-2}(t)]^{[m-(\alpha+2)-1]}%
}{[\Gamma\lbrack m-(\alpha+2)]]_{q}}f(t)d_{\delta}x(t).
\]
And, by induction, we conclude that%

\begin{align}
\delta_{0}^{m}(\delta_{0}^{\alpha-m}f(z))  &  =\int_{a+\frac{1}{2}}%
^{z-\frac{m-(\alpha+m)}{2}}\frac{[x(z)-x_{m-(\alpha+2)-2}(t)]^{[m-(\alpha
+m)-1]}}{[\Gamma\lbrack m-(\alpha+m)]]_{q}}f(t)d_{\delta}x(t)\nonumber\\
&  =\int_{a+\frac{1}{2}}^{z+\frac{\alpha}{2}}\frac{[x(z)-x_{-\alpha
-2}(t)]^{(-\alpha-1)}}{[\Gamma(-\alpha)]_{q}}f(t)d_{\delta}x(t). \label{rld}%
\end{align}

Therefore from Eq. (\ref{rld}), we can give the following equivalent
definition of the $\alpha$-th Riemann-Liouville central fractional difference:

\begin{definition}
\label{rlcentral2}Assume that $\alpha\notin N,$ let $m$ be the smallest
integer exceeding $\alpha>0$, \ the $\alpha$-th Riemann-Liouville central
fractional difference can be defined by%
\begin{equation}
\delta_{0}^{\alpha}f(z)=\int_{a+\frac{1}{2}}^{z+\frac{\alpha}{2}}%
\frac{[x(z)-x_{-\alpha-2}(t)]^{(-\alpha-1)}}{[\Gamma(-\alpha)]_{q}%
}f(t)d_{\delta}x(t), \label{rlc2}%
\end{equation}
where $f(z)$ is defined in $\{a+\frac{1}{2},\operatorname{mod}(1)\},\delta
_{0}^{\alpha}f(z)$ is defined in $\{a+\frac{-\alpha+1}{2},\operatorname{mod}%
(1)\}$.
\end{definition}

We can also give the definition of the Caputo central fractional difference as follows:

\begin{definition}
\label{caputocentral}Let $m$ be the smallest integer exceeding $\alpha$, \ the
$\alpha$-th Caputo central fractional difference is defined by%
\begin{equation}
^{C}\delta_{0}^{\alpha}f(z)=\delta_{0}^{\alpha-m}\delta_{0}^{m}f(z).
\end{equation}

\end{definition}

We should mention that it is also important to establish the analogue Euler
Beta formula on non-uniform lattices with respect to the central fractional sum.

\begin{theorem}
\label{eulerbetac}For any $\alpha,\beta,$we have%
\begin{align}
&  \int_{a+\frac{1}{2}+\frac{\alpha}{2}}^{z-\frac{\beta}{2}}\frac
{[x(z)-x_{\beta-2}(t)]^{(\beta-1)}}{[\Gamma(\beta)]_{q}}\cdot\frac{\lbrack
x(t)-x_{\alpha-1}(a)]^{(\alpha)}}{[\Gamma(\alpha+1)]_{q}}d_{\delta
}x(t)\nonumber\\
&  =\frac{[x(z)-x_{\alpha+\beta-1}(a)]^{(\alpha+\beta)}}{[\Gamma(\alpha
+\beta+1)]_{q}}. \label{central}%
\end{align}

\end{theorem}

\begin{proof}
Since
\[
a+\frac{1}{2}+\frac{\alpha}{2}\leq t\leq z-\frac{\beta}{2},
\]
we have%
\[
a+1\leq t+\frac{1}{2}-\frac{\alpha}{2}\leq z+\frac{1}{2}-\frac{\alpha+\beta
}{2}.
\]
Set
\[%
\genfrac{\{}{.}{0pt}{}{t+\frac{1}{2}-\frac{\alpha}{2}=\overline{t}}{z+\frac
{1}{2}-\frac{\alpha+\beta}{2}=\overline{z}}%
\]
then%

\[%
\genfrac{\{}{.}{0pt}{}{t=\overline{t}-\frac{1}{2}+\frac{\alpha}{2}%
}{z=z-\frac{1}{2}+\frac{\alpha+\beta}{2}}%
.
\]
The LHS of Eq. (\ref{central}) is equivalent to%

\begin{align*}
&  \int_{a+1}^{\overline{z}}\frac{[x(\overline{z}+\frac{\alpha+\beta-1}%
{2})-x_{\beta-2}(\overline{t}+\frac{\alpha+\beta-1}{2})]^{(\beta-1)}}%
{[\Gamma(\beta)]_{q}}\frac{[x(\overline{t}+\frac{\alpha-1}{2})-x_{\alpha
-1}(a)]^{(\alpha)}}{[\Gamma(\alpha+1)]_{q}}d_{\delta}x(\overline{t}%
+\frac{\alpha-1}{2})\\
&  =\int_{a+1}^{\overline{z}}\frac{[x_{\alpha+\beta-1}(\overline{z}%
)-x_{\alpha+\beta-1}(\overline{t}-1)]^{(\beta-1)}}{[\Gamma(\beta)]_{q}}%
\frac{[x_{\alpha-1}(\overline{t})-x_{\alpha-1}(a)]^{(\alpha)}}{[\Gamma
(\alpha+1)]_{q}}d_{\nabla}x_{\alpha}(\overline{t}),
\end{align*}
and the RHS is is equivalent to%
\begin{align*}
&  \frac{[x(\overline{z}+\frac{\alpha+\beta-1}{2})-x_{\alpha+\beta
-1}(a)]^{(\alpha+\beta)}}{[\Gamma(\alpha+\beta+1)]_{q}}\\
&  =\frac{[x_{\alpha+\beta-1}(\overline{z})-x_{\alpha+\beta-1}(a)]^{(\alpha
+\beta)}}{[\Gamma(\alpha+\beta+1)]_{q}}.
\end{align*}
By the use of Euler Beta \textbf{Theorem} \ref{eulerbeta} on non-uniform
lattices, \textbf{Theorem} \ref{eulerbetac} is completed.
\end{proof}

\begin{proposition}
\label{banqunc}For any $\operatorname{Re}\alpha,\operatorname{Re}\beta>0,$we
have%
\begin{equation}
\delta_{0}^{-\beta}\delta_{0}^{-\alpha}f(z)=\delta_{0}^{-(\alpha+\beta)}f(z).
\end{equation}
where $f(z)$ is defined in $\{a+\frac{1}{2},\operatorname{mod}(1)\},\delta
^{-\alpha}f(z)$ is defined in $\{a+\frac{\alpha+1}{2},\operatorname{mod}%
(1)\},$ and $(\delta^{-\beta}\delta^{-\alpha}f)(z)$ is defined in
$\{a+\frac{\alpha+\beta+1}{2},\operatorname{mod}(1)\}.$
\end{proposition}

\begin{proof}
By \textbf{Definition} \ref{sumcentral}, we have
\begin{align*}
&  \int_{a+\frac{1}{2}+\frac{\alpha}{2}}^{z-\frac{\beta}{2}}\frac
{[x(z)-x_{\beta-2}(t)]^{(\beta-1)}}{[\Gamma(\beta)]_{q}}\nabla_{\gamma
}^{-\alpha}f(t)d_{\delta}(t)\\
&  =\int_{a+\frac{1}{2}+\frac{\alpha}{2}}^{z-\frac{\beta}{2}}\frac
{[x(z)-x_{\beta-1}(t)]^{(\beta-1)}}{[\Gamma(\beta)]_{q}}d_{\delta}(t)\\
&  \int_{a+\frac{1}{2}}^{t-\frac{\alpha}{2}}\frac{[x(t)-x_{\alpha
-2}(s)]^{(\alpha-1)}}{[\Gamma(\alpha)]_{q}}f(s)d_{\delta}(s)\\
&  =\int_{a+\frac{1}{2}}^{z-\frac{\alpha+\beta}{2}}f(s)d_{\delta}%
(s)\int_{s+\frac{\alpha}{2}}^{z-\frac{\beta}{2}}\frac{[x(z)-x_{\beta
-2}(t)]^{(\beta-1)}}{[\Gamma(\beta)]_{q}}\\
&  \frac{[x(t)-x_{\alpha-2}(s-1)]^{(\alpha-1)}}{[\Gamma(\alpha)]_{q}}%
d_{\delta}(t).
\end{align*}
In viewer of \textbf{Theorem} \ref{eulerbetac}, one has
\begin{align*}
&  \int_{s+\frac{\alpha}{2}}^{z-\frac{\beta}{2}}\frac{[x(z)-x_{\beta
-2}(t)]^{(\beta-1)}}{[\Gamma(\beta)]_{q}}\frac{[x(t)-x_{\alpha-2}%
(s)]^{(\alpha-1)}}{[\Gamma(\alpha)]_{q}}d_{\delta}(t)\\
&  =\frac{[x(z)-x_{\alpha+\beta-2}(s)]^{(\alpha+\beta-1)}}{[\Gamma
(\alpha+\beta)]_{q}},
\end{align*}
it yields%

\begin{align*}
\delta_{0}^{-\beta}\delta_{0}^{-\alpha}f(z)  &  =\int_{a+\frac{1}{2}}%
^{z-\frac{\alpha+\beta}{2}}\frac{[x(z)-x_{\alpha+\beta-2}(s)]^{(\alpha
+\beta-1)}}{[\Gamma(\alpha+\beta)]_{q}}f(s)\nabla x_{\gamma}(s)\\
&  =\delta_{0}^{-(\alpha+\beta)}f(z).
\end{align*}

\end{proof}

\begin{proposition}
\label{basicc}For any $\operatorname{Re}\alpha>0,$we have%
\[
\delta_{0}^{\alpha}\delta_{0}^{-\alpha}f(z)=f(z).
\]

\end{proposition}

\begin{proof}
By \textbf{Definition} \ref{rlcentral}, we have%
\[
\delta_{0}^{\alpha}\delta_{0}^{-\alpha}f(z)=\delta_{0}^{m}(\delta_{0}%
^{\alpha-m})\delta_{0}^{-\alpha}f(z).
\]
In view of \textbf{Theorem} \ref{banqunc}, one has%
\[
\delta_{0}^{\alpha-m}\delta_{0}^{-\alpha}f(z)=\delta_{0}^{-m}f(z),
\]
it yields that%

\[
\delta_{0}^{\alpha}\delta_{0}^{-\alpha}f(z)=\delta_{0}^{m}\delta_{0}%
^{-m}f(z)=f(z).
\]

\end{proof}

\begin{proposition}
\label{taylorkc} Let $k\in N,$ then%
\begin{equation}
\delta_{0}^{-k}\delta_{0}^{k}f(z)=f(z)-\sum_{j=0}^{k-1}\frac{\delta_{0}%
^{j}f(a)}{[j]_{q}!}[x(z)-x_{j-1}(a)]^{(j)}. \label{ctaylor}%
\end{equation}

\begin{proof}
When $k=1,$ we have%
\begin{align*}
\delta_{0}^{-1}\delta_{0}^{1}f(z)  &  =\sum_{a+\frac{1}{2}}^{z+\frac{1}{2}%
}\delta_{0}^{1}f(s)\delta x(s)\\
&  =\sum_{a+\frac{1}{2}}^{z+\frac{1}{2}}\delta^{1}f(s)=f(z)-f(a).
\end{align*}
Assume that when $n=k,$ (\ref{ctaylor}) holds, then for $n=k+1,$ we conclude
that%
\begin{align*}
\delta_{0}^{-(k+1)}\delta_{0}^{k+1}f(z)  &  =\delta_{0}^{-1}[\delta_{0}%
^{-k}\delta_{0}^{k}]\delta_{0}^{1}f(z)\\
&  =\delta_{0}^{-1}\{\delta_{0}^{1}f(a)-\sum_{j=0}^{k-1}\frac{\delta_{0}%
^{j}[\delta_{0}^{1}f](a)}{[j]_{q}!}[x(z)-x_{j-1}(a)]^{(j)}\\
&  =f(z)-f(a)-\sum_{j=0}^{k-1}\frac{\delta_{0}^{j+1}f(a)}{[j+1]_{q}%
!}[x(z)-x_{j}(a)]^{(j+1)}\\
&  =f(z)-\sum_{j=0}^{k}\frac{\delta_{0}^{j}f(a)}{[j]_{q}!}[x(z)-x_{j-1}%
(a)]^{(j)}.
\end{align*}
Therefore, by the induction, the proof of (\ref{ctaylor}) is completed.
\end{proof}
\end{proposition}

\begin{proposition}
\label{ftaylorc} Let $0<k-1<\alpha\leq k,$ then%
\begin{equation}
\delta_{0}^{-\alpha}\delta_{0}^{\alpha}f(z)=f(z)-\sum_{j=0}^{k-1}\delta
_{0}^{j-k+\alpha}f(a)\frac{[x(z)-x_{\alpha+j-k-1}(a)]^{(j+\alpha-k)}}%
{[\Gamma(j+\alpha-k+1)]_{q}}. \label{p2}%
\end{equation}

\end{proposition}

\begin{proof}
Since%
\[
\delta_{0}^{-\alpha}\delta_{0}^{\alpha}f(z)=\delta_{0}^{-\alpha+k}\delta
_{0}^{-k}\delta_{0}^{k}\delta_{0}^{-k+\alpha}f(z),
\]
then by the use of (\ref{ctaylor}), one has%
\begin{align*}
\delta_{0}^{-\alpha}\delta_{0}^{\alpha}f(z)  &  =\delta_{0}^{-\alpha
+k}\{\delta_{0}^{-k+\alpha}f(z)-\sum_{j=0}^{k-1}\frac{\delta_{0}^{j}\delta
_{0}^{-k+\alpha}f(a)}{[j]_{q}!}[x(z)-x_{j-1}(a)]^{(j)}\\
&  =f(z)-\sum_{j=0}^{k-1}\delta_{0}^{j-k+\alpha}f(a)\frac{[x(z)-x_{\alpha
+j-k-1}(a)]^{(j+\alpha-k)}}{[\Gamma(j+\alpha-k+1)]_{q}}.
\end{align*}

\end{proof}

\begin{proposition}
Let $0<k-1<q\leq k,$ then%
\begin{equation}
\delta_{0}^{-p}\delta_{0}^{q}f(z)=\delta_{0}^{q-p}f(z)-\sum_{j=1}^{k}%
\delta_{0}^{q-j}f(a)\frac{[x(z)-x_{p-j-1}(a)]^{(p-j)}}{[\Gamma(p-j+1)]_{q}}.
\label{p5}%
\end{equation}

\begin{proof}
Since%
\[
\delta_{0}^{-p}\delta_{0}^{q}f(z)=\delta_{0}^{-p+q}\delta_{0}^{-q}\delta
_{0}^{q}f(z),
\]
then by the use of (\ref{p2}), one has%
\begin{align*}
&  \delta_{0}^{-p+q}\delta_{0}^{-q}\delta_{0}^{q}f(z)\\
&  =\delta_{0}^{-p+q}\{f(z)-\sum_{j=1}^{k}\delta_{0}^{q-j}f(a)\frac
{[x(z)-x_{p-j-1}(a)]^{(q-j)}}{[\Gamma(q-j+1)]_{q}}\}\\
&  =\delta_{0}^{-p+q}f(z)-\sum_{j=1}^{k}\delta_{0}^{q-j}f(a)\frac
{[x(z)-x_{p-j-1}(a)]^{(p-j)}}{[\Gamma(p-j+1)]_{q}},
\end{align*}
the equality (\ref{p5}) is completed.
\end{proof}
\end{proposition}

\begin{proposition}
Let $0<k-1<q\leq k,$ $p>0,$ then%
\begin{equation}
\delta_{0}^{p}\delta_{0}^{q}f(z)=\delta_{0}^{p+q}f(z)-\sum_{j=1}^{k}\delta
_{0}^{q-j}f(a)\frac{[x(z)-x_{-p-j-1}(a)]^{(-p-j)}}{[\Gamma(-p-j+1)]_{q}}.
\label{p6}%
\end{equation}

\end{proposition}

\begin{proof}
Let $m-1<p\leq m,$ In view of
\[
\delta_{0}^{p}\delta_{0}^{q}f(z)=\delta_{0}^{m}\delta_{0}^{-m+p}\delta_{0}%
^{q}f(z),
\]
then by the use of (\ref{p5}), we have%
\begin{align*}
&  \delta_{0}^{m}\delta_{0}^{-m+p}\delta_{0}^{q}f(z)\\
&  =\delta_{0}^{m}\{\delta_{0}^{-m+p+q}f(z)-\sum_{j=1}^{k}\delta_{0}%
^{q-j}f(a)\frac{[x(z)-x_{m-p-j-1}(a)]^{(m-p-j)}}{[\Gamma(m-p-j+1)]_{q}}\}\\
&  =\delta_{0}^{p+q}f(z)-\sum_{j=1}^{k}\delta_{0}^{q-j}f(a)\frac
{[x(z)-x_{-p-j-1}(a)]^{(-p-j)}}{[\Gamma(-p-j+1)]_{q}},
\end{align*}
the equality (\ref{p6}) is completed.
\end{proof}

The relationship between Riemann-Liouville fractional difference and Caputo
fractional difference is

\begin{proposition}
\label{relation}We have%
\[
^{C}\delta_{0}^{\alpha}f(z)=\delta_{0}^{\alpha}\{f(z)-\sum_{j=1}^{m-1}%
\frac{\delta_{0}^{j}f(a)}{[j]_{q}!}[x(z)-x_{j-1}(a)]^{(j)}\}.
\]

\end{proposition}

\begin{proof}
According to \textbf{Definition} \ref{caputocentral} and \textbf{Proposition}
\ref{taylorkc}, we have%

\begin{align*}
^{C}\delta_{0}^{\alpha}f(z)  &  =\delta_{0}^{\alpha-m}\delta_{0}%
^{m}f(z)=\delta_{0}^{\alpha}\delta_{0}^{-m}\delta_{0}^{m}f(z)\\
&  =\delta_{0}^{\alpha}\{f(z)-\sum_{j=1}^{m-1}\frac{\delta^{j}f(a)}{[j]_{q}%
!}[x(z)-x_{j-1}(a)]^{(j)}\}.
\end{align*}

\end{proof}

\begin{proposition}
\label{basiccaputo}Let $0<m-1<\alpha\leq m,$ then
\begin{equation}
^{C}\delta_{0}^{\alpha}\delta_{0}^{-\alpha}f(z)=f(z). \label{p3}%
\end{equation}

\begin{proof}
Set%
\[
g(z)=\delta_{0}^{-\alpha}f(z),
\]
then we know that%
\[
g(a)=\delta_{0}g(a)=...=\delta_{0}^{m-1}g(a)=0,
\]
So that from \textbf{Proposition} \ref{relation} we have%
\[
^{C}\delta_{0}^{\alpha}g(z)=\delta_{0}^{\alpha}g(z)=f(z).
\]
Therefore, one has%
\begin{equation}
^{C}\delta_{0}^{\alpha}\delta_{0}^{-\alpha}f(z)=\delta_{0}^{\alpha}\delta
_{0}^{-\alpha}f(z)=f(z).
\end{equation}

\end{proof}
\end{proposition}

\begin{proposition}
\label{fractaylorc}Let $0<m-1<\alpha\leq m,$ then%
\begin{equation}
\delta_{0}^{-\alpha}[^{C}\delta_{0}^{\alpha}]f(z)=f(z)-\sum_{j=1}^{m-1}%
\frac{\delta^{j}f(a)}{[j]_{q}!}[x(z)-x_{j-1}(a)]^{(j)}. \label{p4}%
\end{equation}

\begin{proof}
By \textbf{Definition} \ref{caputocentral}, one has%
\begin{align}
\delta_{0}^{-\alpha}[^{C}\delta_{0}^{\alpha}]f(z)  &  =\delta_{0}^{-\alpha
}\delta_{0}^{-(m-\alpha)}\delta_{0}^{m}f(z)=\delta_{0}^{-m}\delta_{0}%
^{m}f(z)\\
&  =f(z)-\sum_{j=1}^{m-1}\frac{\delta^{j}f(a)}{[j]_{q}!}[x(z)-x_{j-1}%
(a)]^{(j)}.
\end{align}

\end{proof}
\end{proposition}

\begin{definition}
\label{sequencedef}When $0<\alpha\leq1,$we can give a type of sequential
fractional central difference on non-uniform, which is defined as%
\begin{equation}
^{S}\delta_{0}^{k\alpha}f(z)=\underbrace{\delta_{0}^{\alpha}\delta_{0}%
^{\alpha}...\delta_{0}^{\alpha}}f(z).(k-multiple)
\end{equation}

\end{definition}

For sequential fractional central difference on non-uniform, we can obtain
Taylor formula.

\begin{theorem}
\label{sequencetaylor} Let $0<\alpha\leq1,k\in N,$ then%
\begin{equation}
\delta_{0}^{-k\alpha}[^{S}\delta_{0}^{k\alpha}f](z)=f(z)-\sum_{j=0}^{k-1}%
\frac{^{S}\delta_{0}^{j\alpha}f(a)}{[\Gamma(j\alpha+1)]_{q}}[x(z)-x_{j\alpha
-1}(a)]^{j\alpha}. \label{staylor}%
\end{equation}

\end{theorem}

\begin{proof}
When $k=1,$ from \textbf{proposition} \ref{fractaylorc}, we have%
\[
\delta_{0}^{-\alpha}\delta_{0}^{\alpha}f(z)=f(z)-f(a).
\]
Assume that when $n=k,$ (\ref{staylor}) holds, then for $n=k+1,$we conclude
that%
\begin{align*}
r_{k+1}(z)  &  =\delta_{0}^{-(k+1)\alpha}[^{S}\delta_{0}^{(k+1)\alpha
}f(z)=\delta_{0}^{-\alpha}\delta_{0}^{-k\alpha}[^{S}\delta_{0}^{k\alpha
}]\delta_{0}^{\alpha}f(z)\\
&  =\delta_{0}^{-\alpha}\{\delta_{0}^{\alpha}f(a)-\sum_{j=0}^{k-1}\frac
{^{S}\delta_{0}^{j\alpha}[\delta_{0}^{\alpha}f](a)}{[\Gamma(j\alpha+1)]_{q}%
!}[x(z)-x_{j\alpha-1}(a)]^{(j\alpha)}\\
&  =f(z)-f(a)-\sum_{j=0}^{k-1}\frac{^{S}\delta_{0}^{j+1}f(a)}{[\Gamma
((j+1)\alpha+1)]_{q}}[x(z)-x_{j}(a)]^{(j+1)\alpha}\\
&  =f(z)-\sum_{j=0}^{k}\frac{\delta_{0}^{j\alpha}f(a)}{[\Gamma(j\alpha
+1)]_{q}}[x(z)-x_{j\alpha-1}(a)]^{(j\alpha)}.
\end{align*}
Therefore, by the induction, the proof of (\ref{staylor}) is completed.
\end{proof}

\begin{theorem}
\label{sequencetaylor2} The following Taylor series:%
\[
f(z)=\sum_{k=0}^{\infty}[^{S}\delta_{0}^{k\alpha}f](a)\frac{[x(z)-x_{k\alpha
-1}(a)]^{(k\alpha)}}{[\Gamma(k\alpha+1)]_{q}}%
\]
holds if and only if%
\[
\lim_{k\rightarrow\infty}r_{k}(z)=\lim_{k\rightarrow\infty}\delta
_{0}^{-k\alpha}[^{S}\delta_{0}^{k\alpha}]f(z)=0.
\]

\end{theorem}

\begin{proof}
This is a direct consequence of \textbf{Theorem} \ref{sequencetaylor}.
\end{proof}

\section{\bigskip Applications: Series Solution of Fractional Difference
Equations}

Next we will give the the solution of the fractional central difference
equation on nonuniform lattices as follows:%

\begin{equation}
^{C}\delta_{0}^{\alpha}f(z)=\lambda f(z).(0<\alpha\leq1) \label{fra-eq}%
\end{equation}

\begin{theorem}
The solution of Eq.(\ref{fra-eq}) is%
\begin{equation}
f(z)=%
{\displaystyle\sum\limits_{k=0}^{\infty}}
\lambda^{k}\frac{[x(z)-x_{k\alpha-1}(a)]^{(k\alpha)}}{[\Gamma(k\alpha+1)]_{q}%
}.
\end{equation}

\end{theorem}

\begin{proof}
Using the generalized sequence Taylor's series, assuming that the solution
$f(z)$ can be written as%

\begin{equation}
f(z)=%
{\displaystyle\sum\limits_{k=0}^{\infty}}
c_{k}\frac{[x(z)-x_{k\alpha-1}(a)]^{(k\alpha)}}{[\Gamma(k\alpha+1)]_{q}}.
\label{f}%
\end{equation}
From the equality%

\begin{align*}
^{C}\delta_{0}^{\alpha}\frac{[x(z)-x_{k\alpha-1}(a)]^{(k\alpha)}}%
{[\Gamma(k\alpha+1)]_{q}}  &  =[^{C}\delta_{0}^{\alpha}]\delta_{0}^{-k\alpha
}(1)=\delta_{0}^{\alpha-1}\delta_{0}^{1}\delta_{0}^{-k\alpha}(1)\\
&  =\delta_{0}^{\alpha-1}\delta_{0}^{1-k\alpha}(1)=\delta_{0}^{-(k-1)\alpha
}(1)\\
&  =\frac{[x(z)-x_{(k-1)\alpha-1}(a)]^{((k-1)\alpha)}}{[\Gamma((k-1)\alpha
+1)]_{q}},
\end{align*}
we obtain%

\begin{equation}
^{C}\delta_{0}^{\alpha}f(z)=%
{\displaystyle\sum\limits_{k=1}^{\infty}}
c_{k}\frac{[x(z)-x_{k\alpha-1}(a)]^{((k-1)\alpha)}}{[\Gamma((k-1)\alpha
+1)]_{q}}. \label{cf}%
\end{equation}
Substituting (\ref{f}) and (\ref{cf}) into (\ref{fra-eq}) yields%

\begin{equation}%
{\displaystyle\sum\limits_{k=1}^{\infty}}
c_{k+1}\frac{[x(z)-x_{k\alpha-1}(a)]^{(k\alpha)}}{[\Gamma(k\alpha+1)]_{q}%
}-\lambda%
{\displaystyle\sum\limits_{k=0}^{\infty}}
c_{k}\frac{[x(z)-x_{k\alpha-1}(a)]^{(k\alpha)}}{[\Gamma(k\alpha+1)]_{q}}=0.
\label{eq}%
\end{equation}
Equating the coefficient of $[x(z)-x_{k\alpha-1}(a)]^{(k\alpha)}$ to zero in
(\ref{eq}), we get%
\begin{equation}
c_{k+1}=\lambda c_{k},
\end{equation}
that is%

\[
c_{k}=\lambda^{k}c_{0}.
\]
Therefore, we obtain the solution of (\ref{fra-eq}) is%

\[
f(z)=c_{0}%
{\displaystyle\sum\limits_{k=0}^{\infty}}
\lambda^{k}\frac{[x(z)-x_{k\alpha-1}(a)]^{(k\alpha)}}{[\Gamma(k\alpha+1)]_{q}%
}.
\]

\end{proof}

\begin{definition}
\label{qexp}The basic $\alpha-$order fractional exponential function is
defined by%
\begin{equation}
e(\alpha,z)=%
{\displaystyle\sum\limits_{k=0}^{\infty}}
\frac{[x(z)-x_{k\alpha-1}(a)]^{(k\alpha)}}{[\Gamma(k\alpha+1)]_{q}},
\label{fraexp}%
\end{equation}
and%
\begin{equation}
e(\alpha,\lambda,z)=%
{\displaystyle\sum\limits_{k=0}^{\infty}}
\lambda^{k}\frac{[x(z)-x_{k\alpha-1}(a)]^{(k\alpha)}}{[\Gamma(k\alpha+1)]_{q}%
}. \label{fraexp2}%
\end{equation}

\end{definition}

\begin{remark}
When $\alpha=1$ in (\ref{fraexp2}), the basic $1-$order fractional exponential
function on a $q-$quadritic lattices was originally introduced by Ismail,
Zhang \cite{ismail1994}, and Suslov \cite{suslov2003} with different notation
and normalization, which was very important for Basic Fourier analytic.
Definition \ref{qexp} is an natural extension of it.
\end{remark}

\begin{example}
Let us consider a general $n\alpha-$order sequence fractional difference
equation with coefficients on nonuniform lattices of the form:%
\begin{equation}
\lbrack a_{n}(^{S}\delta^{n\alpha})+a_{n-1}(^{S}\delta^{(n-1)\alpha
})+...+a_{1}(^{S}\delta^{\alpha})+a_{0}(^{S}\delta^{0})]f(z)=0 \label{eq2}%
\end{equation}

\end{example}

\begin{proof}
As in the classical case, substituing
\[
f(z)=e(\alpha,\lambda,z),
\]
into Eq.(\ref{eq2}), one can obtain%

\begin{equation}
a_{n}\lambda^{n}+a_{n-1}\lambda^{n-1}+...+a_{1}\lambda+a_{0}=0. \label{eq21}%
\end{equation}
Assume Eq.(\ref{eq21}) have different roots $\lambda_{i},i=1,2,...,n$, then
one can get $n$ linearly independent solutions%
\[
f_{i}(z)=e(\alpha,\lambda_{i},z),i=1,2,...,n
\]

\end{proof}

\begin{example}
Let $\omega>0,$ consider $2\alpha-$order sequence fractional difference
equation for harmonic motion of the form:%
\begin{equation}
^{S}\delta_{0}^{\alpha S}\delta_{0}^{\alpha}f(z)+\omega^{2}f(z)=0,(0<\alpha
\leq1) \label{eq3}%
\end{equation}
and its solutions are related to the generalized basic trigonometric functions.
\end{example}

\begin{proof}
Set
\begin{equation}
f(z)=e(\alpha,\lambda,z), \label{e}%
\end{equation}
substituing (\ref{e}) into Eq.(\ref{eq3}), then we have%

\[
\lambda^{2}+\omega^{2}=0,
\]
which has two solutions%

\[
\lambda_{1}=i\omega,\lambda_{2}=-i\omega.
\]
So the soluutions of Eq.(\ref{eq2}) are%

\begin{align*}
f_{1}(z)  &  =e(\alpha,i\omega,z)=%
{\displaystyle\sum\limits_{k=0}^{\infty}}
(i\omega)^{k}\frac{[x(z)-x_{k\alpha-1}(a)]^{(k\alpha)}}{[\Gamma(k\alpha
+1)]_{q}}\\
&  =%
{\displaystyle\sum\limits_{n=0}^{\infty}}
(-1)^{n}\omega^{2n}\frac{[x(z)-x_{2n\alpha-1}(a)]^{(2n\alpha)}}{[\Gamma
(2n\alpha+1)]_{q}}+\\
&  +i%
{\displaystyle\sum\limits_{n=0}^{\infty}}
(-1)^{n}\omega^{2n+1}\frac{[x(z)-x_{(2n+1)\alpha-1}(a)]^{((2n+1)\alpha)}%
}{[\Gamma((2n+1)\alpha+1)]_{q}},
\end{align*}
and%

\begin{align*}
f_{2}(z)  &  =e(\alpha,-i\omega,z)=%
{\displaystyle\sum\limits_{k=0}^{\infty}}
(-i\omega)^{k}\frac{[x(z)-x_{k\alpha-1}(a)]^{(k\alpha)}}{[\Gamma
(k\alpha+1)]_{q}}\\
&  =%
{\displaystyle\sum\limits_{n=0}^{\infty}}
(-1)^{n}\omega^{2n}\frac{[x(z)-x_{2n\alpha-1}(a)]^{(2n\alpha)}}{[\Gamma
(2n\alpha+1)]_{q}}+\\
&  -i%
{\displaystyle\sum\limits_{n=0}^{\infty}}
(-1)^{n}\omega^{2n+1}\frac{[x(z)-x_{(2n+1)\alpha-1}(a)]^{((2n+1)\alpha)}%
}{[\Gamma((2n+1)\alpha+1)]_{q}}.
\end{align*}
Using the notation of Euler, we denote%

\[
\cos(\alpha,\omega,z)=%
{\displaystyle\sum\limits_{n=0}^{\infty}}
(-1)^{n}\omega^{2n}\frac{[x(z)-x_{2n\alpha-1}(a)]^{(2n\alpha)}}{[\Gamma
(2n\alpha+1)]_{q}},
\]
and%
\[
\sin(\alpha,\omega,z)=%
{\displaystyle\sum\limits_{n=0}^{\infty}}
(-1)^{n}\omega^{2n+1}\frac{[x(z)-x_{(2n+1)\alpha-1}(a)]^{((2n+1)\alpha)}%
}{[\Gamma((2n+1)\alpha+1)]_{q}}.
\]
Then it holds that%
\begin{align*}
\cos(\alpha,\omega,z)  &  =\frac{e(\alpha,i\omega,z)+e(\alpha,-i\omega,z)}%
{2},\\
\sin(\alpha,\omega,z)  &  =\frac{e(\alpha,i\omega,z)-e(\alpha,-i\omega,z)}%
{2i},
\end{align*}
and%
\[
\cos^{2}(\alpha,\omega,z)+\sin^{2}(\alpha,\omega,z)=e(\alpha,i\omega
,z)e(\alpha,-i\omega,z).
\]

\end{proof}

\end{document}